\documentclass[12pt,reqno]{amsart}

\usepackage[utf8]{inputenc}
\usepackage[T1]{fontenc}
\usepackage{amsmath,amssymb,amsthm,mathtools,verbatim}
\usepackage{enumitem}
\usepackage[colorlinks=true,linkcolor=blue,citecolor=blue,urlcolor=blue]{hyperref}

\theoremstyle{plain}
\newtheorem{theorem}{Theorem}[section]
\newtheorem{proposition}[theorem]{Proposition}
\newtheorem{lemma}[theorem]{Lemma}
\newtheorem{corollary}[theorem]{Corollary}

\theoremstyle{definition}
\newtheorem{definition}[theorem]{Definition}

\theoremstyle{remark}
\newtheorem{remark}[theorem]{Remark}

\newcommand{\N}{\mathbb N}
\newcommand{\R}{\mathbb R}

\newcommand{\K}{\mathcal K}
\newcommand{\A}{\mathcal A}
\newcommand{\Wiener}{\mathcal W}
\newcommand{\FBV}{\mathcal F^{\mathrm{BV}}}
\newcommand{\Finf}{\mathcal F^{\infty}}
\newcommand{\HS}{\mathcal S_2}
\newcommand{\Comm}{\mathfrak C}

\newcommand{\PV}{\operatorname{p.v.}}
\newcommand{\thetaM}{\widetilde\vartheta}
\newcommand{\eps}{\varepsilon}
\newcommand{\norm}[1]{\left\lVert #1\right\rVert}

\newcommand{\ol}[1]{\overline{#1}}
\DeclareMathOperator{\Var}{Var}

\DeclareMathOperator{\spec}{spec}

\newcommand{\SigmaC}{\Sigma_C}

\newcommand{\Toep}{\mathcal T}
\newcommand{\Comp}{\mathcal K}
\newcommand{\Kernels}{\mathcal A}

\numberwithin{equation}{section}
\allowdisplaybreaks

\title[A discrete Mellin calculus in the Toeplitz algebra]
      {A discrete Mellin calculus in the Toeplitz algebra}

\author{Carlo Bellavita}
\address{Dipartimento di Matematica ``F. Enriques'', Dipartimento di
Eccellenza MUR 2023--2027, Universit\`a degli Studi di Milano, Via
C.\ Saldini 50, 20133 Milano, Italy}
\email{carlo.bellavita@gmail.com}
\email{carlo.bellavita@unimi.it}

\author{Georgios Stylogiannis}
\address{Department of Mathematics, Aristotle University of
Thessaloniki, 54124 Thessaloniki, Greece}
\email{g.stylog@gmail.com}
\email{stylog@math.auth.gr}

\subjclass[2020]{Primary 47B35; Secondary 47B10, 44A35}
\keywords{Toeplitz algebra, Ces\`aro operator, Mellin transform,
Calkin algebra, Hilbert--Schmidt operator, multiplicative convolution}

\thanks{The first author is a member of the Gruppo Nazionale per
l'Analisi Matematica, la Probabilit\`a e le loro Applicazioni
(GNAMPA) of the Istituto Nazionale di Alta Matematica (INdAM)}

\begin{document}

\begin{abstract}
We prove that the classical Ces\`aro operator belongs to the Toeplitz
algebra, providing an independent solution to a question raised by
Barr\'ia and Halmos. Our approach is based on a discrete Mellin
calculus for the sampled-ratio matrices
\[
  W(\kappa)_{jk} = \frac{1}{j+1}\, \kappa\!\left(\frac{k+1}{j+1}\right).
\]
For a natural algebra of kernels $\Kernels$, we prove that this
quantization is multiplicative modulo Hilbert--Schmidt operators,
\[
  W(\kappa)W(\eta) - W(\kappa \star \eta) \in \mathcal S_2,
  \qquad \kappa,\eta \in \Kernels.
\]
We further show that every operator $W(\kappa)$, $\kappa \in
\Kernels$, belongs to the commutator ideal of the Toeplitz algebra.
Since the Ces\`aro operator corresponds to the kernel $\kappa =
\mathbf 1_{(0,1]}$, this resolves the Barr\'{\i}a--Halmos question as
a special case of the general framework.
\end{abstract}

\maketitle

\tableofcontents
\section{Introduction}
Let \(H^2\) denote the Hardy space of the unit disc, with standard
orthonormal basis \(e_n(z)=z^n\), \(n\ge0\).  For
\(b\in L^\infty(\mathbb T)\), the Toeplitz operator with symbol \(b\)
is defined by
\[
  T_b=PM_b|_{H^2},
\]
where \(M_b\) is multiplication by \(b\) on \(L^2(\mathbb T)\) and
\(
  P:L^2(\mathbb T)\longrightarrow H^2
\)
is the orthogonal projection.

We denote by
\[
  \Toep=C^*\{T_b:b\in L^\infty(\mathbb T)\}
\]
the Toeplitz algebra on \(H^2\), that is, the norm-closed \(C^*\)-algebra
generated by all Toeplitz operators.  The algebra \(\Toep\) contains
the ideal \(\Comp=\mathcal K(H^2)\) of compact operators
\cite{BH}.  We write
\[
  \pi:B(H^2)\longrightarrow B(H^2)/\Comp
\]
for the Calkin map and denote by \(\Comm\) the commutator ideal of
\(\Toep\).

Modulo \(\Comm\), every element of \(\Toep\) is represented by a
Toeplitz operator: more precisely, each \(T\in\Toep\) can be written
in the form
\[
  T=T_b+Q,
  \qquad b\in L^\infty(\mathbb T),\quad Q\in\Comm;
\]
see, for example, \cite{F}.

The starting point of this paper is a question of Barr\'{\i}a and
Halmos concerning the classical Ces\`aro operator
\begin{equation}\label{eq:cesaro-def}
  Ce_k = \sum_{j \ge k} \frac{1}{j+1}\, e_j, \qquad k \ge 0.
\end{equation}
In \cite{BH} they asked whether $C$ belongs to the Toeplitz algebra.
They also proved that $C$ is not a finite sum of finite products of
Toeplitz operators. Thus a positive solution cannot come from a finite
algebraic identity in Toeplitz operators: the norm closure in the
definition of $\Toep$ has to enter in an essential way. The problem
was still recorded as open in the recent monograph of Mashreghi and
Ross \cite{MR}. Our first result answers this question.

\begin{theorem}[Barr\'{\i}a--Halmos question]\label{thm:cesaro-main}
The classical Ces\`aro operator belongs to the commutator ideal of the
Toeplitz algebra. In particular,
\[
  C \in \Comm \subset \Toep.
\]
\end{theorem}

\subsection{The sampled-ratio family}
Rather than seeking a direct representation of the Ces\`aro operator
in terms of Toeplitz operators, we embed it into the larger family of
sampled-ratio matrices
\begin{equation}\label{eq:Wkappa-def}
  W(\kappa)_{jk} = \frac{1}{j+1}\, \kappa\!\left(\frac{k+1}{j+1}\right),
  \qquad j,k \ge 0,
\end{equation}
where $\kappa$ is a pointwise defined function.
At this stage $W(\kappa)$ is regarded only as a pointwise-defined
matrix; it need not define a bounded operator on $H^2$. The Ces\`aro
operator is precisely the member of this family corresponding to
\[
  \omega_C = \mathbf 1_{(0,1]}, \qquad C = W(\omega_C).
\]
Thus the original membership problem for $C$ becomes part of a more
general question: when does a sampled Mellin matrix $W(\kappa)$ belong
to the Toeplitz algebra?

For comparison, consider the continuous dilation operator on
$L^2(0,\infty)$,
\[
  (W_c(\kappa) f)(x) = \int_0^\infty \kappa(u)\, f(xu)\, du.
\]
For suitable kernels, operator composition is represented exactly by
Mellin convolution,
\[
  W_c(\kappa) W_c(\eta) = W_c(\kappa \star \eta), \qquad
  (\kappa \star \eta)(w) = \int_0^\infty \kappa(u)\, \eta\!\left(\frac{w}{u}\right) \frac{du}{u}.
\]
After passing to logarithmic coordinates
\begin{equation}\label{eq:log-coords}
  (Uf)(t) = e^{t/2} f(e^t), \qquad f_\kappa(t) = e^{t/2}\kappa(e^t),
\end{equation}
the operators $W_c(\kappa)$ become translation-convolution operators,
\[
  (UW_c(\kappa)U^{-1} g)(t) = \int_{\mathbb R} f_\kappa(s)\, g(t+s)\, ds.
\]
With the Fourier convention
$\widehat h(\xi) = \int_{\mathbb R} h(t) e^{i\xi t}\, dt$, this operator
is diagonalized by the multiplier $\widehat{f_\kappa}(-\xi)$;
equivalently, up to the harmless reflection $\xi \mapsto -\xi$, the
corresponding Mellin symbol is
\[
  \sigma_\kappa(\xi) = \int_0^\infty \kappa(u)\, u^{-1/2+i\xi}\, du.
\]
Thus the continuous calculus is exactly multiplicative and is
diagonalized by the Mellin transform. The situation changes after
discretization: for the sampled-ratio matrices $W(\kappa)$, exact
multiplicativity generally fails. One of the main points of this paper
is to construct an algebra of kernels $\Kernels$ for which
\[
  W(\kappa) W(\eta) - W(\kappa \star \eta) \in \mathcal S_2, \qquad
  \kappa, \eta \in \Kernels,
\]
so that the multiplicative structure is recovered modulo
Hilbert--Schmidt operators $\mathcal{S}_2$. For $m \ge 0$ we introduce weighted
bounded-variation classes $\FBV_m$, controlled by the weight
$\mu_m(t) = (1+|t|)^m e^{-|t|/2}$, and set $\Kernels = \bigcup_{m\ge0}
\FBV_m$; the precise definition is given in Section~\ref{sec:algebra}.
We also set $\Sigma_C = \{w \in \mathbb C : |w-1| = 1\}$.

\subsection{The discrete Mellin--Toeplitz calculus}
The next theorem summarizes our main result. In particular, it implies
that $\pi(C)$ is normal with spectrum $\Sigma_C$, so that the
continuous functional calculus for $\pi(C)$ is available.

\begin{theorem}[Discrete Mellin--Toeplitz calculus]\label{thm:intro-main}
The class $\Kernels$ is a commutative complex algebra for
multiplicative convolution and is closed under the involution
$\kappa^\vee(u) = \overline{\kappa(1/u)}/u$. Every $W(\kappa)$, $\kappa \in
\Kernels$, defines a bounded operator on $H^2$, and
\begin{equation}\label{eq:intro-HS}
  W(\kappa)W(\eta) - W(\kappa \star \eta) \in \mathcal S_2, \qquad
  \kappa,\eta \in \Kernels.
\end{equation}
Moreover every $W(\kappa)$ belongs to $\Comm \subset \Toep$. The
Mellin symbol extends to an isometric $*$-isomorphism
\[
  \Lambda : C_0(\mathbb R) \longrightarrow \mathcal I_C \subset \pi(\Toep),
  \qquad \Lambda(\sigma_\kappa) = \pi(W(\kappa)),\ \kappa \in \Kernels.
\]
After identifying $\spec(\pi(C)) = \Sigma_C$, the range is
\[
  \mathcal I_C = \{F(\pi(C)) : F \in C(\Sigma_C),\ F(0) = 0\}.
\]
Consequently,
\[
  \|W(\kappa)\|_{\mathrm{ess}} = \|\sigma_\kappa\|_\infty, \qquad
  \spec_{\mathrm{ess}}(W(\kappa)) = \overline{\sigma_\kappa(\mathbb R)},
\]
and $[W(\kappa),W(\eta)] \in \mathcal S_2$,
$[W(\kappa)^*,W(\kappa)] \in \mathcal S_2$.
\end{theorem}

Theorem~\ref{thm:cesaro-main} is the special case $\kappa = \omega_C$
of Theorem~\ref{thm:intro-main}. The discrete Mellin calculus is
therefore not an auxiliary reformulation of the Ces\`aro problem, but
the structural mechanism through which its solution is obtained.

\subsection{Overview of the argument}
The proof splits into two parts of a genuinely different nature, and
it is worth outlining both before entering the technical sections.

The \emph{regular} part of the argument (Sections~\ref{sec:mellin-model}--\ref{sec:Calkin})
constructs the algebra $\Kernels$ and proves~\eqref{eq:intro-HS}. The
difference $W(\kappa)W(\eta) - W(\kappa\star\eta)$ is rewritten
entry by entry as a lattice sum minus an integral, and a localized
Euler--Stieltjes estimate controls this difference by weighted
pointwise variation. The resulting bounds are square-summable, which
gives~\eqref{eq:intro-HS} and produces a $C_0(\mathbb R)$ Mellin
calculus in the Calkin algebra.

At the end of this first part one has only constructed an abstract
Calkin algebra of sampled Mellin operators: nothing yet places it
inside the Calkin image of the Toeplitz algebra. This is where the
\emph{singular} part of the argument (Sections~\ref{sec:L3}--\ref{sec:barria-halmos})
takes over. Its key object is the Toeplitz operator $T_b$ associated
with the sawtooth symbol $b$. A Toeplitz--Hankel semi-commutator
computation, combined with a principal-value intertwining relation for
the singular kernel underlying $T_b$, shows that
\[
  \Lambda(C_0(\mathbb R)) \subset \pi(\Toep).
\]
Since $\pi(C) \in \Lambda(C_0(\mathbb R))$, this gives $C \in \Toep$,
and a final application of the Toeplitz symbol map upgrades this to
the stronger conclusion $C \in \Comm$. Section~\ref{sec:full-calculus}
then collects the consequences of the full construction, stated
already as Theorem~\ref{thm:intro-main} above, and
Section~\ref{sec:concluding} closes with some natural questions left
open by the method.

The continuous background belongs to classical Mellin/Wiener--Hopf and
singular-integral theory; see, for example, Duduchava~\cite{Duduchava}.
On the Toeplitz side, the Gohberg--Krupnik local symbol calculus and
its later treatments involve a hyperbolic Mellin fibre for
piecewise-continuous symbols~\cite{GK,BSK}. There are also notions of
discrete Mellin convolution of a divisor-based type~\cite{Plaschinsky}.
These theories provide important background for the present
construction, but they do not directly yield the sampled-ratio
Hilbert--Schmidt defect estimates required here.

\subsection{Relation with Sang's result}\label{sec:sang}
A recent preprint of Sang~\cite{Sang} gives another affirmative
solution to the Barr\'{\i}a--Halmos question by a different method. The
argument developed in the present paper is independent of Sang's
approach. Sang proves an exact identity of the form
\[
  C = (I+S^*)\, g(T_{\chi_{T_-}}), \qquad
  g(u) = \left[1 + \frac{2i}{\pi}\operatorname{arctanh}(2u-1)\right]^{-1},
\]
with $g(0) = g(1) = 0$, $T_- = \{e^{i\theta} : \pi < \theta <
2\pi\}$, and $S^*$ the backward shift operator. Since $T_{\chi_{T_-}}$ is self-adjoint with spectrum $[0,1]$,
continuous functional calculus gives $C \in \Toep$; the endpoint
condition places $C$ in the commutator ideal as well.

The two approaches solve the same membership problem by genuinely
different mechanisms. Sang's formula is an exact and particularly
short representation of the single operator $C$. In the present paper,
by contrast, the Ces\`aro problem is used as the entry point to a
discrete quantization in which multiplication is no longer exact. The
central analytic statement is the family-level formula
\[
  W(\kappa)W(\eta) = W(\kappa\star\eta) \pmod{\mathcal S_2}, \qquad
  \kappa,\eta \in \Kernels.
\]
The following consequences do not follow  directly from Sang's single-operator identity:
\begin{itemize}
  \item the isometric $*$-isomorphism $\Lambda : C_0(\mathbb R) \to
    \mathcal I_C$, identifying the Calkin classes of \emph{all}
    $W(\kappa)$, $\kappa \in \Kernels$, at once;
  \item the essential norm and essential spectrum formulas
    $\|W(\kappa)\|_{\mathrm{ess}} = \|\sigma_\kappa\|_\infty$,
    $\spec_{\mathrm{ess}}(W(\kappa)) = \overline{\sigma_\kappa(\mathbb R)}$, valid
    for every $\kappa \in \Kernels$, not only for $\kappa = \omega_C$;
  \item the Hilbert--Schmidt commutator estimates $[W(\kappa),W(\eta)]$,\\
    $[W(\kappa)^*,W(\kappa)] \in \mathcal S_2$ for arbitrary pairs in
    $\Kernels$;
  \item the analogous essential norm and spectrum computation for the
    classical Hilbert matrix, obtained later from the same Mellin
    calculus (Corollary~\ref{cor:cesaro-hilbert}).
\end{itemize}
Thus the present argument leads not only to $C \in \Toep$, but to a
family-level Calkin calculus and to the membership of every
$W(\kappa)$, $\kappa \in \Kernels$, in the commutator ideal.

\part{Discrete Mellin calculus}
\section{From the Ces\`aro matrix to the discrete Mellin model}\label{sec:mellin-model}

We now begin the proof of Theorem~\ref{thm:cesaro-main}, following the
route outlined in the introduction. This section embeds the Ces\`aro
matrix into the sampled family~\eqref{eq:Wkappa-def}, records the
Mellin operations governing this family, and introduces the three
kernels --- $\omega_C$, $\omega_\Gamma$, and $\omega_b$ --- that will
later connect it with Toeplitz and Hankel operators.

For Borel functions $\kappa,\eta : (0,\infty) \to \mathbb C$ define
multiplicative convolution by
\begin{equation}\label{eq:mult-conv}
  (\kappa \star \eta)(w) = \int_0^\infty \kappa(u)\, \eta\!\left(\frac{w}{u}\right) \frac{du}{u},
\end{equation}
whenever the integral is absolutely convergent. The Mellin symbol is
\begin{equation}\label{eq:mellin-symbol}
  \sigma_\kappa(\xi) = \int_0^\infty \kappa(u)\, u^{-1/2+i\xi}\, du,
  \qquad \xi \in \mathbb R.
\end{equation}
We use the involution
\begin{equation}\label{eq:involution}
  \kappa^\vee(u) = \frac{\overline{\kappa(1/u)}}{u}.
\end{equation}
A direct entrywise calculation shows that the formal adjoint matrix of
$W(\kappa)$ is $W(\kappa^\vee)$: indeed,

   \[
   \overline{W(\kappa)_{kj}}
   =\frac1{k+1}\overline{\kappa\!\left(\frac{j+1}{k+1}\right)}
   =W(\kappa^\vee)_{jk}.
   \]

so that, whenever $W(\kappa)$ defines a bounded operator,
\begin{equation}\label{eq:adjoint}
  W(\kappa)^* = W(\kappa^\vee).
\end{equation}

Introduce the Wiener class
\begin{equation}\label{eq:wiener-class}
  \Wiener = \left\{ \kappa : \|\kappa\|_{\Wiener} = \int_0^\infty
  |\kappa(u)|\, u^{-1/2}\, du < \infty \right\}.
\end{equation}

\begin{proposition}[Logarithmic dictionary]\label{prop:dictionary}
The map $\kappa \mapsto f_\kappa$ from~\eqref{eq:log-coords} identifies
$\Wiener$, modulo equality almost everywhere, isometrically with
$L^1(\mathbb R)$. More precisely,
\[
  \|\kappa\|_{\Wiener} = \|f_\kappa\|_{L^1}, \qquad
  f_{\kappa \star \eta} = f_\kappa * f_\eta, \qquad
  \sigma_\kappa(\xi) = \widehat{f_\kappa}(\xi) = \int_{\mathbb R} f_\kappa(t) e^{i\xi t}\, dt.
\]
Moreover $f_{\kappa^\vee}(t) = \overline{f_\kappa(-t)}$ and
$\sigma_{\kappa^\vee} = \overline{\sigma_\kappa}$.
\end{proposition}

\begin{proof}
Use $u = e^t$ in~\eqref{eq:wiener-class} and~\eqref{eq:mellin-symbol}.
The convolution identity follows from the same change of variables
in~\eqref{eq:mult-conv}. For the involution, $f_{\kappa^\vee}(t) =
e^{t/2}\kappa^\vee(e^t) = e^{-t/2}\overline{\kappa(e^{-t})} =
\overline{f_\kappa(-t)}$, which is immediate from~\eqref{eq:involution}
since $e^{-t/2}$ is real. Taking Fourier transforms,
\[
  \sigma_{\kappa^\vee}(\xi) = \int_{\mathbb R} \overline{f_\kappa(-t)}\, e^{i\xi t}\, dt
  = \overline{\int_{\mathbb R} f_\kappa(-t)\, e^{-i\xi t}\, dt}
  = \overline{\int_{\mathbb R} f_\kappa(s)\, e^{i\xi s}\, ds} = \overline{\sigma_\kappa(\xi)},
\]
using the substitution $s=-t$.
\end{proof}

We now record three kernels that occur naturally in the argument,
\begin{equation}\label{eq:three-kernels}
  \omega_C = \mathbf 1_{(0,1]}, \qquad
  \omega_\Gamma(u) = \frac{1}{1+u}, \qquad
  \omega_b(u) = \frac{1}{2\pi i(1-u)}\ (u \ne 1),
\end{equation}
with the prescribed point value $\omega_b(1) = 0$. The first kernel
gives exactly the Ces\`aro operator,
\[
  W(\omega_C)_{jk} = \begin{cases} \dfrac{1}{j+1}, & k \le j, \\[4pt] 0, & k > j, \end{cases}
\]
and the second gives a shifted Hilbert matrix,
\begin{equation}\label{eq:hilbert-shift}
  W(\omega_\Gamma)_{jk} = \frac{1}{j+k+2}.
\end{equation}
Thus, if $\Gamma_{jk} = (j+k+1)^{-1}$, then $\Gamma - W(\omega_\Gamma)
\in \mathcal S_2$. For the third kernel, consider the sawtooth
function $b(e^{i\theta}) = 1 - \theta/(2\pi)$, $0 < \theta < 2\pi$,
whose Fourier coefficients are $\widehat b(0) = 1/2$, $\widehat b(n) =
(2\pi i n)^{-1}$ for $n \ne 0$. Hence
\begin{equation}\label{eq:Tb-def}
  T_b = W(\omega_b) + \tfrac12 I.
\end{equation}

The first two kernels belong to $\Wiener$, and their Mellin symbols
are
\begin{equation}\label{eq:symbol-C-Gamma}
\begin{aligned}
  \sigma_{\omega_C}(\xi) &= \int_0^1 u^{-1/2+i\xi}\, du = \frac{1}{1/2+i\xi},\\
  \sigma_{\omega_\Gamma}(\xi) &= \int_0^\infty \frac{u^{-1/2+i\xi}}{1+u}\, du = \pi \operatorname{sech}(\pi\xi).
  \end{aligned}
\end{equation}
The kernel $\omega_b$ is singular at $u=1$ and does not belong to
$\Wiener$; its Mellin symbol is understood in the principal-value
sense,
\begin{equation}\label{eq:symbol-b}
  \PV \int_0^\infty \omega_b(u)\, u^{-1/2+i\xi}\, du = -\tfrac12 \tanh(\pi\xi).
\end{equation}
Indeed, with $s = 1/2+i\xi$,
\[
  \PV \int_0^\infty \frac{u^{s-1}}{1-u}\, du = \int_0^1 \frac{u^{s-1}-u^{-s}}{1-u}\, du
  = \psi(1-s) - \psi(s) = \pi \cot(\pi s),
\]
and~\eqref{eq:symbol-b} follows from the reflection formula for the
digamma function $\psi$. Finally, the Ces\`aro Mellin symbol has the
geometric interpretation
\begin{equation}\label{eq:cesaro-symbol-geom}
  \sigma_{\omega_C} : \mathbb R \longrightarrow \Sigma_C \setminus \{0\},
  \qquad \sigma_{\omega_C}(\xi) = \frac{1}{1/2+i\xi},
\end{equation}
which is a homeomorphism. This last fact will be the geometric anchor
of the Calkin calculus built in Section~\ref{sec:Calkin}.

\section{Boundedness and the discrete Mellin product defect}\label{sec:boundedness}

Having fixed the relevant kernels, we now turn to the two estimates on
which the whole discrete calculus rests: a sufficient condition for
boundedness of a sampled Mellin matrix, and a quantitative comparison
between discrete matrix multiplication and continuous Mellin
convolution. Both estimates will be applied, in
Section~\ref{sec:algebra}, to the specific kernel algebra $\Kernels$
introduced there.

The Wiener condition~\eqref{eq:wiener-class} alone does not imply
boundedness of the sampled matrix: kernels concentrated in a shrinking
interval around $1$ may have uniformly bounded Wiener norm while
$W(\kappa)_{00} = \kappa(1)$ becomes arbitrarily large. We therefore
supplement Wiener integrability with an asymptotic variation
condition.

For a pointwise-defined function $\phi$ on an interval $I$, we denote
by $\Var_I \phi$ its usual total variation
\[
  \Var_I\phi
  =
  \sup_N\left\{
    \sum_{j=1}^N
    |\phi(x_j)-\phi(x_{j-1})|,  \;
    x_0<\cdots<x_N,\;
    x_j\in I
  \right\}.
\]
In particular, isolated
point excursions are counted. For $\kappa \in \Wiener$, set
\[
  g_\kappa(u) = |\kappa(u)|\, u^{-1/2}, \qquad h_\kappa(u) = |\kappa(u)|\, u^{3/2}.
\]
We say that $\kappa$ satisfies condition~\textbf{(R)} if
\begin{equation}\label{eq:condition-R}
  \Var_{[1/M,\infty)} g_\kappa = o(M), \qquad
  \Var_{(0,M]} h_\kappa = o(M), \qquad M \to \infty.
\end{equation}

\begin{lemma}[Calkin--Schur estimate]\label{lem:schur}
If $\kappa \in \Wiener$ satisfies~\textbf{(R)}, then $W(\kappa)$
defines a bounded operator on $H^2$, and
\begin{equation}\label{eq:schur-bound}
  \|\pi(W(\kappa))\| \le \|\kappa\|_{\Wiener}.
\end{equation}
\end{lemma}

\begin{proof}
Put $m_j = j+1$, $L = \|\kappa\|_{\Wiener}$, $p_j = m_j^{-1/2}$. We use
the elementary estimate
\begin{equation}\label{eq:euler-elem}
  q \sum_{r \ge 1} \phi(rq) \le \int_q^\infty \phi(t)\, dt + q\, \Var_{[q,\infty)} \phi,
\end{equation}
valid for every nonnegative integrable $\phi$ of bounded variation on
$[q,\infty)$. Indeed, on each interval $[rq,(r+1)q]$,
\(
 \phi(rq)
 \le
 \phi(t)+\Var_{[rq,(r+1)q]}\phi
\)
,for every $t\in[rq,(r+1)q]$ and integration followed by summation gives
\eqref{eq:euler-elem}.

For the $j$-th row put $a = m_j$ and $c=m_k$. Then
\[
  R_j = \frac{1}{p_j} \sum_{k\ge0} |W(\kappa)_{jk}|\, p_k =a^{1/2}\sum_{c\geq 1}\frac{1}{a} \left\vert\kappa\left(\frac{c}{a}\right)\right\vert\frac{1}{c^{1/2}}= \frac1a \sum_{c\ge1} g_\kappa\!\left(\frac ca\right),
\]
and~\eqref{eq:euler-elem} with $q=1/a$ gives $R_j \le L + \eps_j$,
$\eps_j = m_j^{-1} \Var_{[1/m_j,\infty)} g_\kappa \to 0$.

For the $k$-th column put $b = m_k$ and $F_b(x) = b^{1/2} x^{-3/2}
|\kappa(b/x)| = b^{-1} h_\kappa(b/x)$. The change of variables $u=b/x$
gives $\int_0^\infty F_b = L$, and since $x \mapsto b/x$ maps
$[1,\infty)$ monotonically onto $(0,b]$, $\Var_{[1,\infty)} F_b = b^{-1}
\Var_{(0,b]} h_\kappa$. Consequently
\[
\begin{aligned}
    C_k &= \frac{1}{p_k} \sum_{j\ge0} |W(\kappa)_{jk}|\, p_j=
 b^{1/2}\sum_{j\ge0}
 \frac{1}{m_j}
 \left|\kappa\!\left(\frac b{m_j}\right)\right|\frac{1}{m_j^{1/2}}\\
 &\le \frac{1}{b}\sum_{j\ge1}
 h_\kappa\left(\frac{b}{j}\right)=\sum_{j\geq 1} F_b(j)\le L + \eps'_k,
  \qquad \eps'_k = m_k^{-1} \Var_{(0,m_k]} h_\kappa \to 0.
  \end{aligned}
\]
The weighted Schur test \cite{Schur04, DK03} therefore gives boundedness of $W(\kappa)$.

To obtain the Calkin estimate, let $E_N$ be the orthogonal projection
onto $\operatorname{span}\{e_0,\dots,e_{N-1}\}$ and set $W_N = (I-E_N)
W(\kappa) (I-E_N)$. Since $W(\kappa)-W_N$ has finite rank,
$\pi(W_N) = \pi(W(\kappa))$. Applying the row and column bounds only
to indices $j,k \ge N$ gives
\[
  \|\pi(W(\kappa))\| \le \|W_N\| \le \Bigl(L+\sup_{j\ge N}\eps_j\Bigr)^{1/2}
  \Bigl(L+\sup_{k\ge N}\eps'_k\Bigr)^{1/2},
\]
and letting $N\to\infty$ proves~\eqref{eq:schur-bound}.
\end{proof}

We next compare the product of two sampled matrices with the matrix
associated with their Mellin convolution. The discrepancy is a
Riemann-sum error, localized at the natural lattice scale. For two
kernels $\kappa,\eta$ and $\beta>0$, define
\begin{equation}\label{eq:Hbeta}
  H_\beta(s) = \frac{\kappa(s)\, \eta(\beta/s)}{s}, \qquad s>0,
\end{equation}
and set $H_\beta(s)=0$ for $s\le0$. For $a\ge1$, put
\begin{equation}\label{eq:VaIa}
  V_a(\beta) = \Var_{[1/a,\infty)} H_\beta, \qquad
  I_a(\beta) = \int_0^{1/a} |H_\beta(s)|\, ds.
\end{equation}

The following elementary estimate isolates the only
summation fact used in the lattice argument.

\begin{lemma}[Euler--Stieltjes estimate]\label{lem:euler-stieltjes}
Let \(F\in L^1(\mathbb R)\cap BV(\mathbb R)\) be identically zero on
\((-\infty,1)\). Then \(F(t)\to0\) as \(t\to\pm\infty\), the series
\(
  \sum_{c\ge1}F(c)
\)
converges absolutely, and
\begin{equation}\label{eq:euler-stieltjes}
  \left|
    \sum_{c\ge1}F(c)
    -
    \int_1^\infty F(t)\,dt
  \right|
  \le
  \Var_{[1,\infty)}F.
\end{equation}
\end{lemma}

\begin{proof}
Since \(F\in BV(\mathbb R)\), the limits of \(F(t)\) as
\(t\to\pm\infty\) exist and are finite; since \(F\in L^1(\mathbb R)\),
both limits must be zero.

For \(N\ge1\),
\[
\begin{aligned}
  \sum_{c=1}^N F(c)-\int_1^{N+1}F(t)\,dt
  &=
  \sum_{c=1}^N
  \int_c^{c+1}\bigl(F(c)-F(t)\bigr)\,dt.
\end{aligned}
\]
Hence
\[
\begin{aligned}
  &\left|
    \sum_{c=1}^N F(c)-\int_1^{N+1}F(t)\,dt
  \right|
  \le
  \sum_{c=1}^N
  \int_c^{c+1}|F(c)-F(t)|\,dt\\
  &\qquad \le
  \sum_{c=1}^N
  \Var_{[c,c+1]}F=
  \Var_{[1,N+1]}F\le
  \Var_{[1,\infty)}F.
\end{aligned}
\]

It remains to justify absolute convergence of the sampled series.
Applying the same estimate to \(|F|\), and using
\(
  \Var_I|F|\le\Var_I F
\),
gives
\[
\begin{aligned}
  \sum_{c=1}^N|F(c)|
  &\le
  \int_1^{N+1}|F(t)|\,dt
  +
  \Var_{[1,N+1]}|F|\le
  \|F\|_{L^1(\mathbb R)}
  +
  \Var_{[1,\infty)}F.
\end{aligned}
\]
Thus \(\sum_{c\ge1}|F(c)|<\infty\). Letting \(N\to\infty\) in the
finite-sum estimate proves \eqref{eq:euler-stieltjes}.

Notice that the argument uses the pointwise total variation, so
isolated excursions at sampled points are automatically included.
\end{proof}

\begin{theorem}[Localized lattice estimate]\label{thm:lattice}
Assume that \(W(\kappa)\) and \(W(\eta)\) are bounded. Fix $j,l\ge0$ and set $a=j+1$, $b=l+1$, $\beta=b/a$. Suppose that $H_\beta \in L^1(0,\infty)$ and $V_a(\beta)<\infty$. Then
\begin{equation}\label{eq:lattice-est}
  \left|
    [W(\kappa)W(\eta)]_{jl}
    -
    \frac1a(\kappa\star\eta)(\beta)
  \right|
  \le
  \frac{V_a(\beta)}{a^2}
  +
  \frac{I_a(\beta)}a.
\end{equation}
If these hypotheses hold for every \(a,b\ge1\) and the square of the
right-hand side is summable over \((a,b)\), then
\[
  W(\kappa)W(\eta)-W(\kappa\star\eta)\in\mathcal S_2.
\]
\end{theorem}

\begin{proof}
Define $G(t) = \kappa(t/a)\eta(b/t)/t = a^{-1} H_\beta(t/a)$ for $t>0$,
and $G(t)=0$ for $t\le0$. Since $W(\kappa)$ and $W(\eta)$ are bounded,
the relevant row and column belong to $\ell^2$.
By the definition of the sampled matrices,
\[
\begin{aligned}
 [W(\kappa)W(\eta)]_{jl}
 &=\sum_{k\ge0}W(\kappa)_{jk}W(\eta)_{kl}=\frac1a\sum_{k\ge0}
   \frac1{k+1}
   \kappa\!\left(\frac{k+1}{a}\right)
   \eta\!\left(\frac b{k+1}\right)\\
&=
 \frac1a\sum_{c\ge1}
 \frac1c
 \kappa\!\left(\frac ca\right)
 \eta\!\left(\frac bc\right)= \frac1a \sum_{c\ge1} G(c).
\end{aligned} 
\]
Since
$
\frac1a(\kappa\star\eta)(\beta) =
\frac1a\int_0^\infty G(t)\,dt$, we obtain that
\begin{equation}\label{eq:lattice-diff}
  [W(\kappa)W(\eta)]_{jl} - \frac1a(\kappa\star\eta)(\beta)
  = \frac1a \left( \sum_{c\ge1} G(c) - \int_0^\infty G(t)\,dt \right).
\end{equation}
The interval $(0,1)$ contains no positive lattice point, and $t=as$
gives $\left|\int_0^1 G(t)\,dt\right| \le I_a(\beta)$.
For the lattice error on $[1,\infty)$ set $G_1 = G\,\mathbf 1_{[1,\infty)}$.
The hypotheses give $G_1 \in L^1(\mathbb R)\cap BV(\mathbb R)$, since
$\int_{\mathbb R}|G_1| = \int_{1/a}^\infty |H_\beta(s)|\,ds < \infty$ and
$\Var_{[1,\infty)} G = a^{-1} V_a(\beta)$. 
Applying Lemma~\ref{lem:euler-stieltjes} to \(G_1\) therefore gives
\[
  \left|
    \sum_{c\ge1}G(c)
    -
    \int_1^\infty G(t)\,dt
  \right|
  \le
  \frac1a V_a(\beta).
\]
Together with \eqref{eq:lattice-diff} and
\(
  \left|\int_0^1G(t)\,dt\right|
  \le
  I_a(\beta)
\),
this proves \eqref{eq:lattice-est}.
Square summability of the resulting
matrix entries gives the Hilbert--Schmidt conclusion.
\end{proof}

\begin{remark}\label{rem:variation-pointwise}
The variation term in~\eqref{eq:lattice-est} is genuinely pointwise:
isolated excursions of $H_\beta$ at sampled points contribute to total
variation even though they are invisible to the distributional
derivative. This is why the regular kernel class introduced in the
next section is defined using a variation measure that remembers the
chosen pointwise representative.
\end{remark}

\section{The regular kernel algebra and Hilbert--Schmidt multiplicativity}\label{sec:algebra}

We now use the estimates of Section~\ref{sec:boundedness} to build the
algebra of kernels $\Kernels$ announced in the introduction, and to
prove the Hilbert--Schmidt product formula~\eqref{eq:intro-HS}. For
$m\ge0$ put
\begin{equation}\label{eq:weight}
  \mu_m(t) = (1+|t|)^m e^{-|t|/2}, \qquad t \in \mathbb R.
\end{equation}
We first define a variation measure that remembers the chosen
pointwise representative, in the spirit of Remark~\ref{rem:variation-pointwise}.

\begin{definition}[Pointwise variation]\label{def:pointwise-var}
Let $f$ be locally of bounded variation and continuous outside a
finite set $J$. For $\tau \in J$ define the excess excursion
\[
  E_f(\tau) = |f(\tau)-f(\tau^-)| + |f(\tau^+)-f(\tau)| - |f(\tau^+)-f(\tau^-)| \ge 0,
\]
and set
\begin{equation}\label{eq:Df}
  |Df| = |df| + \sum_{\tau\in J} E_f(\tau)\, \delta_\tau,
\end{equation}
where $df$ is the distributional derivative of $f$, viewed as a Radon
measure, and $|df|$ is its total variation.
\end{definition}

Thus $|Df|$ records both the distributional variation and the extra
variation caused solely by the chosen value at an exceptional point:
if $f(0^-)=f(0^+)=1$ but $f(0)=2$, then $df$ has no atom at $0$,
whereas $|Df|$ has an atom of mass $2$. 

With the convention that endpoint atoms are included in the measure
of a compact interval,
\begin{equation}\label{eq:variation-measure}
  \Var_{[a,b]}f\le |Df|([a,b])
\end{equation}
for every compact interval $[a,b]$.  The same estimate extends to
half-lines.  Indeed,
\[
  \Var_{[a,\infty)}f
  =
  \sup_{b>a}\Var_{[a,b]}f
  \le
  \sup_{b>a}|Df|([a,b])
  =
  |Df|([a,\infty)),
\]
and similarly
\(\Var_{(-\infty,b]}f
  \le
  |Df|((-\infty,b])\).


We record the pointwise product
estimate associated with this variation.

\begin{lemma}[Pointwise product variation]\label{lem:leibniz}
Let $u,v$ be pointwise BV functions, continuous outside finite
exceptional sets, and write $r^\#(t) = \max\{|r(t^-)|,|r(t)|,|r(t^+)|\}$.
Then
\begin{equation}\label{eq:leibniz}
  |D(uv)| \le v^\# |Du| + u^\# |Dv|.
\end{equation}
\end{lemma}
\begin{proof}
Let \(J\) be the union of the exceptional sets of \(u\) and \(v\).

On each component of \(\mathbb{R}\setminus J\), both functions are continuous
and of bounded variation. Hence the usual Lebesgue--Stieltjes product
rule gives
\(
  d(uv)=u\,dv+v\,du
\),
and therefore
\[
  |d(uv)|
  \le
  |v|\,|du|+|u|\,|dv|
  \le
  v^\#\,|Du|+u^\#\,|Dv|.
\]

It remains to check the atoms at the points of \(J\). Fix
\(\tau\in J\), and write
\(
  u_-=u(\tau^-),\quad u_0=u(\tau),\quad u_+=u(\tau^+)
\),
and similarly \(v_-,v_0,v_+\). By the definition of the pointwise
variation measure,
\[
  |Du|(\{\tau\})
  =
  |u_0-u_-|+|u_+-u_0|,
\]
and analogously for \(v\). Moreover,
\[
\begin{aligned}
 |D(uv)|(\{\tau\})
 &=
 |u_0v_0-u_-v_-|
 +
 |u_+v_+-u_0v_0|.
\end{aligned}
\]
For the first term,
\[
\begin{aligned}
 |u_0v_0-u_-v_-|
 &\le
 |v_0|\,|u_0-u_-|
 +
 |u_-|\,|v_0-v_-|\\
 &\le
 v^\#(\tau)|u_0-u_-|
 +
 u^\#(\tau)|v_0-v_-|,
\end{aligned}
\]
while
\[
\begin{aligned}
 |u_+v_+-u_0v_0|
 &\le
 v^\#(\tau)|u_+-u_0|
 +
 u^\#(\tau)|v_+-v_0|.
\end{aligned}
\]
Adding these two inequalities gives
\[
 |D(uv)|(\{\tau\})
 \le
 v^\#(\tau)|Du|(\{\tau\})
 +
 u^\#(\tau)|Dv|(\{\tau\}).
\]
Combining the nonatomic part with the finitely many atoms proves
\eqref{eq:leibniz}.
\end{proof}

\begin{definition}[Kernel classes]\label{def:kernel-classes}
A pointwise-defined kernel $\kappa$ belongs to $\FBV_m$ if its
logarithmic profile $f_\kappa$ is locally of bounded variation and
there exist $K>0$ and a finite set $J\subset\mathbb R$ such that
$f_\kappa$ is continuous on $\mathbb R\setminus J$ and
\begin{equation}\label{eq:BV-bound}
  |f_\kappa(t)| \le K\mu_m(t), \qquad
  |Df_\kappa| \le K\mu_m(t)\Bigl(dt + \sum_{\tau\in J}\delta_\tau\Bigr).
\end{equation}
Set $\Kernels = \bigcup_{m\ge0} \FBV_m$.
\end{definition}

\begin{lemma}[One-sided values]\label{lem:onesided}
If $\kappa \in \FBV_m$ satisfies~\eqref{eq:BV-bound} with constant
$K$, then $f_\kappa$ has finite one-sided limits at every $\tau \in
\mathbb R$, and $|f_\kappa(\tau^\pm)| \le K\mu_m(\tau)$.
\end{lemma}

\begin{proof}
Local bounded variation gives existence of both one-sided limits. If
$t_n>\tau$, $t_n\downarrow\tau$, then $|f_\kappa(t_n)|\le K\mu_m(t_n)$;
passing to the limit and using continuity of $\mu_m$ yields the
right-hand estimate. The left-hand estimate is identical.
\end{proof}

The Ces\`aro and Hilbert kernels belong to this class: 
\[
f_{\omega_C}(t)
= e^{t/2}\mathbf 1_{(-\infty,0]}(t),\, f_{\omega_\Gamma}(t) = \tfrac12
\operatorname{sech}(t/2)\quad \text{ so } \omega_C,\omega_\Gamma \in \FBV_0.
\]
We need three stability facts.

\begin{lemma}[Convolution of the weights]\label{lem:weight-conv}
For $m,n\ge0$ there exists $C_{m,n}>0$ such that
\begin{equation}\label{eq:weight-conv}
  (\mu_m * \mu_n)(t) \le C_{m,n}\, \mu_{m+n+1}(t), \qquad t \in \mathbb R.
\end{equation}
\end{lemma}

\begin{proof}
Let $J_t$ be the closed segment joining $0$ and $t$. Then $|s|+|t-s| =
|t| + 2\operatorname{dist}(s,J_t)$, so the exponential part of the
convolution is $e^{-|t|/2} e^{-\operatorname{dist}(s,J_t)}$. On $J_t$
the polynomial factors are $O((1+|t|)^{m+n})$ and the interval has
length $|t|$; on the two complementary rays the exponential factor
absorbs all polynomial growth. This proves~\eqref{eq:weight-conv}.
\end{proof}

\begin{proposition}[Algebra and involution]\label{prop:closure}
If $\kappa \in \FBV_m$ and $\eta \in \FBV_n$, then $\kappa\star\eta
\in \FBV_{m+n+1}$. Moreover $\kappa^\vee \in \FBV_m$. Hence
$\Kernels$ is a commutative complex $\star$-algebra closed under
$\vee$.
\end{proposition}

\begin{proof}
Write $f=f_\kappa$, $g=f_\eta$. By~\eqref{eq:BV-bound}, $f,g\in
L^1(\mathbb R)$, so Proposition~\ref{prop:dictionary} gives $f_{\kappa\star\eta}
= f*g$, and the pointwise bound follows from~\eqref{eq:BV-bound} and
Lemma~\ref{lem:weight-conv}.

We next control the variation of $f*g$. Distributionally,
$d(f*g)=f*dg$. The convolution is continuous because
$\|(f*g)(\cdot+h)-f*g\|_\infty \le \|f(\cdot+h)-f\|_1 \|g\|_\infty \to
0$. Moreover, the convolution of the absolutely continuous measure
$f(t)\,dt$ with the finite measure $dg$ is absolutely continuous:
$d(f*g)(t) = h(t)\,dt$, $h(t) = \int_{\mathbb R} f(t-s)\,dg(s)$. Using~\eqref{eq:BV-bound},
\[
  |h(t)| \le C(\mu_m*\mu_n)(t) + C\sum_{\tau\in J_g} \mu_n(\tau)\mu_m(t-\tau)
  \le C'\mu_{m+n+1}(t),
\]
by Lemma~\ref{lem:weight-conv}, the finiteness of $J_g$, and the
elementary bound $\mu_m(t-\tau)\le C_{\tau,m}\mu_{m+n+1}(t)$. Hence
$|d(f*g)|\le C'\mu_{m+n+1}(t)\,dt$; since $f*g$ is continuous, it has
no point excursions, and $\kappa\star\eta \in \FBV_{m+n+1}$ follows.

Finally, $f_{\kappa^\vee}(t)=\overline{f_\kappa(-t)}$ by
Proposition~\ref{prop:dictionary}; since $\mu_m$ is even and
$|\overline z| = |z|$, reflection and conjugation preserve both the
pointwise bound and the variation estimate in~\eqref{eq:BV-bound}
(conjugation is an isometry of $\mathbb C$, so it leaves $|Df|$
unchanged). Finite linear combinations are handled using
subadditivity of pointwise variation and the finite union of the
exceptional sets. This proves that $\Kernels$ is a commutative complex
$\star$-algebra.
\end{proof}

\begin{lemma}[Regularity]\label{lem:regularity}
Every $\kappa\in\FBV_m$ belongs to $\Wiener$ and satisfies $(R)$.
More precisely, if \eqref{eq:BV-bound} holds with constant $K$ and
exceptional set $J$, then, for $M\ge1$,
\[
  \max\left\{
    \Var_{[1/M,\infty)}g_\kappa,\,
    \Var_{(0,M]}h_\kappa
  \right\}
  \le
  C_{K,m,J}
  +
  4K\sqrt M(1+\log M)^m,
\]
where one may take
\[
  C_{K,m,J}=
  2K\int_0^\infty
       (1+t)^m e^{-3t/2}\,dt
  +
  K\sum_{\tau\in J}
       (1+|\tau|)^m e^{|\tau|/2}.
\]
\end{lemma}

\begin{proof}
The Wiener estimate follows from $\|\kappa\|_{\Wiener} = \|f_\kappa\|_1$
and~\eqref{eq:BV-bound}. For the first variation, use $u=e^t$ and set
$G(t)=g_\kappa(e^t)=e^{-t}|f_\kappa(t)|$. Since $u=e^t$ is increasing,
$\Var_{[1/M,\infty)} g_\kappa = \Var_{[-\log M,\infty)} G$. 
Since the map \(z\mapsto |z|\) is \(1\)-Lipschitz, taking absolute
values does not increase the pointwise variation. More precisely,
\(
  |D|f_\kappa||\le |Df_\kappa|
\).
Indeed, this is standard for the nonatomic part of the variation, while
at an exceptional point \(\tau\),
\[
\begin{aligned}
 &\bigl||f_\kappa(\tau)|-|f_\kappa(\tau^-)|\bigr|
 +\bigl||f_\kappa(\tau^+)|-|f_\kappa(\tau)|\bigr|\\
 &\qquad\le
 |f_\kappa(\tau)-f_\kappa(\tau^-)|
 +|f_\kappa(\tau^+)-f_\kappa(\tau)|.
\end{aligned}
\]


Hence, by the pointwise Leibniz estimate,
\[
  |DG|
  \le
  e^{-t}|D|f_\kappa||
  +
  e^{-t}|f_\kappa|^\#(t)\,dt.
\]
By \eqref{eq:BV-bound} and Lemma~\ref{lem:onesided},
\[
  |f_\kappa|^\#(t)
  =
  \max\bigl\{
    |f_\kappa(t^-)|,\,
    |f_\kappa(t)|,\,
    |f_\kappa(t^+)|
  \bigr\}
  \le K\mu_m(t).
\]
Together with
\(
  |D|f_\kappa||\le |Df_\kappa|
\),
this gives
\[
  |DG|
  \le
  K e^{-t}\mu_m(t)
  \left(
    2\,dt+\sum_{\tau\in J}\delta_\tau
  \right).
\]

Therefore, using the half-line version of
\eqref{eq:variation-measure},
\[
\begin{aligned}
  \Var_{[1/M,\infty)}g_\kappa
  &=
  \Var_{[-\log M,\infty)}G\\
  &\le
  2K\int_{-\log M}^{\infty}
       e^{-t}\mu_m(t)\,dt
  +
  K\sum_{\substack{\tau\in J\\
                    \tau\ge-\log M}}
       e^{-\tau}\mu_m(\tau).
\end{aligned}
\]


For the contribution of $[0,\infty)$,
\[
  2K\int_0^\infty e^{-t}\mu_m(t)\,dt
  =
  2K\int_0^\infty
       (1+t)^m e^{-3t/2}\,dt.
\]
Moreover,
\[
  e^{-\tau}\mu_m(\tau)
  \le
  (1+|\tau|)^m e^{|\tau|/2},
  \qquad \tau\in\mathbb R,
\]
so the atomic contribution is bounded by
\(
  K\sum_{\tau\in J}
  (1+|\tau|)^m e^{|\tau|/2}
\).
Thus both contributions are bounded by $C_{K,m,J}$.

On the negative interval,
\[
\begin{aligned}
  2K\int_{-\log M}^{0}e^{-t}\mu_m(t)\,dt
  &=
  2K\int_0^{\log M}(1+s)^m e^{s/2}\,ds\\
  &\le
  4K\sqrt M(1+\log M)^m.
\end{aligned}
\]
Hence
\[
  \Var_{[1/M,\infty)}g_\kappa
  \le
  C_{K,m,J}
  +
  4K\sqrt M(1+\log M)^m.
\]

For the second variation we use the involution.  By
Proposition~\ref{prop:closure}, $\kappa^\vee\in\FBV_m$, the defining
estimate is preserved with the same constant $K$, and the exceptional
set $J$ is replaced by $-J$.  Moreover,
\[
  g_{\kappa^\vee}(u)
  =
  |\kappa^\vee(u)|u^{-1/2}
  =
  |\kappa(1/u)|u^{-3/2}
  =
  h_\kappa(1/u).
\]
Since $u\mapsto1/u$ maps $[1/M,\infty)$ monotonically onto $(0,M]$,
\(
  \Var_{(0,M]}h_\kappa
  =
  \Var_{[1/M,\infty)}g_{\kappa^\vee}
\).
The constant introduced above is invariant under reflection of the
exceptional set,
\(
  C_{K,m,-J}=C_{K,m,J}
\),
and therefore the first estimate applied to $\kappa^\vee$ gives
\[
  \Var_{(0,M]}h_\kappa
  \le
  C_{K,m,J}
  +
  4K\sqrt M(1+\log M)^m.
\]
\end{proof}

By Lemmas~\ref{lem:schur} and~\ref{lem:regularity}, every $W(\kappa)$,
$\kappa \in \Kernels$, defines a bounded operator; we can therefore
apply the localized lattice estimate of Theorem~\ref{thm:lattice} to
kernels in $\Kernels$.

\begin{theorem}[Hilbert--Schmidt product formula]\label{thm:HSproduct}
Let $\kappa \in \FBV_m$ and $\eta \in \FBV_n$. Then
\begin{equation}\label{eq:HSproduct}
  W(\kappa)W(\eta) - W(\kappa\star\eta) \in \mathcal S_2.
\end{equation}
More precisely, there are $C>0$ and $q=q(m,n)>0$ --- for instance
$q=2(m+n+1)$ --- such that, with $a=j+1$, $b=l+1$,
\begin{equation}\label{eq:HSproduct-quant}
  |[W(\kappa)W(\eta) - W(\kappa\star\eta)]_{jl}|
  \le C\, \frac{(1+\log a)^q (1+\log b)^q}{ab}.
\end{equation}
\end{theorem}

\begin{proof}

Let \(f=f_\kappa\) and \(g=f_\eta\).  Choose constants \(K,K'>0\)
and finite exceptional sets \(J_f,J_g\) for which the defining
estimates \eqref{eq:BV-bound} hold for \(f\) and \(g\), respectively.
Throughout the proof, the constant \(C\) may depend on
\(
  m,\ n,\ K,\ K',\ J_f,\ J_g
\)
but never on the lattice parameters \(a\) and \(b\).
Set
\(
  \beta= b/a\), \(y=\log\beta
\).

By~\eqref{eq:Hbeta} and
the definition of the logarithmic profiles,
\begin{equation}\label{eq:Hbeta-log}
  H_\beta(e^x) = e^{-y/2-x} f(x)\, g(y-x).
\end{equation}
In particular, $\int_0^\infty |H_\beta(s)|\,ds = e^{-y/2}\int_{\mathbb R}
|f(x)g(y-x)|\,dx \le e^{-y/2}\|f\|_\infty\|g\|_1 < \infty$. 
{By Lemmas~\ref{lem:regularity} and~\ref{lem:schur},
\(W(\kappa)\) and \(W(\eta)\) are bounded.  We will estimate the two
quantities in \eqref{eq:lattice-est}; the variation estimate will
in particular show that \(V_a(\beta)<\infty\), so that Theorem~\ref{thm:lattice} applies.}

Assume first that $b\le a$, so $y\le0$. Set $L=1+\log a$, $B=1+\log
b$; then $B\le L$ and $1+|y|\le L$.

\emph{Tail term.} With $s=e^x$ and $x=y-w$,
\[
  I_a(\beta) = e^{-y/2} \int_{\log b}^\infty |f(y-w)g(w)|\,dw.
\]
Using~\eqref{eq:BV-bound} and $1+w+|y| \le L(1+w)$,
\[
  I_a(\beta) \le CL^m \int_{\log b}^\infty (1+w)^{m+n} e^{-w}\,dw
  \le C\, \frac{L^m B^{m+n}}{b},
\]
so that
\begin{equation}\label{eq:Ia-bound}
  \frac{I_a(\beta)}{a} \le C\, \frac{L^m B^{m+n}}{ab}.
\end{equation}

\emph{Variation term.} Write $H_\beta(e^x) = e^{-y/2} U(x)V(x)$,
$U(x)=e^{-x}f(x)$, $V(x)=g(y-x)$. By Lemmas~\ref{lem:onesided}
and~\ref{lem:leibniz}, the nonatomic contribution is bounded by
\[
  Ce^{-y/2} \int_{-\log a}^\infty e^{-x}\mu_m(x)\mu_n(y-x)\,dx.
\]
Using $|x|+|y-x| = |y|+2\operatorname{dist}(x,[y,0])$, the factor
$e^{-y/2}=e^{|y|/2}$ cancels the common factor $e^{-|y|/2}$ from the
two weights. Thus it remains to estimate
\[
  Q=
  \int_{-\log a}^{\infty}
  e^{-x}(1+|x|)^m(1+|y-x|)^n
  e^{-\text{dist}(x,[y,0])}\,dx.
\]
Since \(b\le a\), we have
\(
  -\log a\le y\le0
\).
We split \(Q=Q_1+Q_2+Q_3\) over
\(
  [-\log a,y], [y,0], [0,\infty)
\).

On \([-\log a,y]\), put \(x=y-w\). Then
\(0\le w\le\log b\),
\(
  |x|=|y|+w\le\log a
\),
and
\(
  e^{-x}e^{-\text{dist}(x,[y,0])}
  =
  e^{-y+w}e^{-w}
  =
  e^{-y}
  =
  \frac ab
\).
Therefore
\[
\begin{aligned}
 Q_1
 &\le
 C\frac ab L^m
 \int_0^{\log b}(1+w)^n\,dw\le
 C\frac ab L^mB^{n+1}.
\end{aligned}
\]

On \([y,0]\), write \(x=y+u\), \(0\le u\le |y|\). Since
\(\text{dist}(x,[y,0])=0\),
\(
  e^{-x}=e^{-y}e^{-u}=\frac ab e^{-u}
\),
while both polynomial factors are bounded by \(CL^m(1+u)^n\).
Hence
\[
  Q_2
  \le
  C\frac ab L^m
  \int_0^\infty(1+u)^n e^{-u}\,du
  \le
  C\frac ab L^m.
\]

Finally, for \(x\ge0\),
\(
  \text{dist}(x,[y,0])=x
\),
so the exponential factor is \(e^{-2x}\). Since
\(
  1+|y-x|
  =
  1+x+|y|
  \le
  L(1+x)
\), we have that
\[
  Q_3
  \le
  CL^n\int_0^\infty
  (1+x)^{m+n}e^{-2x}\,dx
  \le
  CL^n.
\]
It remains to estimate the finitely many atoms arising from the
exceptional sets of \(f\) and \(g\). 
{For simplicity we sum over the whole exceptional sets \(J_f\) and \(J_g\), rather than only over the points lying in the relevant half-line; this only enlarges the resulting upper bound.}
By
\eqref{eq:BV-bound} and Lemma~\ref{lem:onesided}, their total
contribution is bounded by
\[
\begin{aligned}
 &Ce^{-y/2}
 \sum_{\tau\in J_f}
 e^{-\tau}\mu_m(\tau)\mu_n(y-\tau)+
 Ce^{-y/2}
 \sum_{\sigma\in J_g}
 e^{-(y-\sigma)}\mu_m(y-\sigma)\mu_n(\sigma).
\end{aligned}
\]
Since \(J_f\) is finite, for each fixed \(\tau\in J_f\),
\[
 e^{-y/2}e^{-\tau}\mu_m(\tau)\mu_n(y-\tau)
 \le
 C_\tau(1+|y|)^n
 \le
 C_\tau L^n.
\]
Similarly, for each fixed \(\sigma\in J_g\),
\[
 e^{-y/2}e^{-(y-\sigma)}
 \mu_m(y-\sigma)\mu_n(\sigma)
 \le
 C_\sigma e^{-y}(1+|y|)^m
 \le
 C_\sigma\frac ab L^m.
\]
Therefore, after summing over the finite exceptional sets,
\[
  \text{atomic contribution}
  \le
  C\left(
    L^n+\frac abL^m
  \right).
\]

Consequently,
\begin{equation}\label{eq:Va-bound}
  \frac{V_a(\beta)}{a^2} \le C \left( \frac{L^m B^{n+1}}{ab} + \frac{L^m}{ab} + \frac{L^n}{a^2} \right).
\end{equation}

{Combining \eqref{eq:Ia-bound}, and \eqref{eq:Va-bound} with \eqref{eq:lattice-est}, all terms
are already of the required form except for
\(
  C\dfrac{L^n}{a^2}
\).
Here the assumption \(b\le a\) is used: since
\(
  \dfrac1{a^2}\le \dfrac1{ab},
\)
and \(L,B\ge1\), this term is also bounded by
\(
  C\dfrac{L^qB^q}{ab}
\).
Taking, for instance,
\(
  q=2(m+n+1)
\), gives \eqref{eq:HSproduct-quant}.}

Suppose now that \(b>a\), and set
\[
  \Delta(\kappa,\eta)
  =
  W(\kappa)W(\eta)-W(\kappa\star\eta).
\]
Using
\(
  W(\kappa)^*=W(\kappa^\vee)\), and \(
  (\kappa\star\eta)^\vee
  =
  \eta^\vee\star\kappa^\vee
\),
we obtain
\(
  \Delta(\kappa,\eta)^*
  =
  \Delta(\eta^\vee,\kappa^\vee)
\).
By Proposition~\ref{prop:closure},
\(
  \eta^\vee\in\FBV_n\), \(\kappa^\vee\in\FBV_m\),
and
\(
  |\Delta(\kappa,\eta)_{jl}|
  =
  |\Delta(\eta^\vee,\kappa^\vee)_{lj}|.
\)
After exchanging \(j\) and \(l\), the second lattice parameter does
not exceed the first, so the estimate proved above applies to the
ordered pair \((\eta^\vee,\kappa^\vee)\), with \(m\) and \(n\)
interchanged.

The resulting constant may differ from the one obtained in the first
case, but it has the same admissible dependence on the defining data.
Since
\[
  q(m,n)=2(m+n+1)=q(n,m),
\]
the exponent is unchanged.  Replacing \(C\) by the larger of the two
constants therefore yields \eqref{eq:HSproduct-quant} for all \(a,b\ge1\).

Suppose now that \(b>a\). Set
\[
  \Delta(\kappa,\eta)
  =
  W(\kappa)W(\eta)-W(\kappa\star\eta).
\]
Using
\(
  W(\kappa)^*=W(\kappa^\vee)
\)
and
\(
  (\kappa\star\eta)^\vee
  =
  \eta^\vee\star\kappa^\vee
\),
we obtain the exact identity
\(
  \Delta(\kappa,\eta)^*
  =
  W(\eta^\vee)W(\kappa^\vee)
  -
  W(\eta^\vee\star\kappa^\vee)
\).
By Proposition~\ref{prop:closure},
\(
  \eta^\vee\in \FBV_n,  \kappa^\vee\in \FBV_m.
\)
Moreover,
\(
  |\Delta(\kappa,\eta)_{jl}|
  =
  |\Delta(\eta^\vee,\kappa^\vee)_{lj}|
\).
Thus exchanging \(j\) and \(l\) reduces the estimate to the case in
which the second lattice parameter does not exceed the first. Hence the preceding estimate applies and gives
\eqref{eq:HSproduct-quant}.

Finally, if $\Delta = W(\kappa)W(\eta)-W(\kappa\star\eta)$,
then~\eqref{eq:HSproduct-quant} gives
\[
  \sum_{j,l\ge0} |\Delta_{jl}|^2 \le C^2 \left( \sum_{r\ge1} \frac{(1+\log r)^{2q}}{r^2} \right)^2 < \infty,
\]
so $\Delta \in \mathcal S_2$.
\end{proof}

\begin{corollary}\label{cor:commutators}
For $\kappa,\eta \in \Kernels$,
\[
  [W(\kappa),W(\eta)] \in \mathcal S_2, \qquad
  [W(\kappa)^*,W(\kappa)] \in \mathcal S_2.
\]
\end{corollary}

\begin{proof}
Subtract~\eqref{eq:HSproduct} for the ordered pairs $(\kappa,\eta)$ and
$(\eta,\kappa)$, using commutativity of $\star$. For the second
statement take $\eta=\kappa^\vee$ and use~\eqref{eq:adjoint}.
\end{proof}

\section{The Calkin Mellin calculus}\label{sec:Calkin}
We begin by deriving the Calkin-algebra consequences of the boundedness and Hilbert--Schmidt estimates established above.

\begin{proposition}[Contractive symbol calculus]\label{prop:contractive}
For $\kappa\in\A$, the Calkin class $\pi(W(\kappa))$ is normal and
\begin{equation}\label{eq:contractive}
  \norm{\pi(W(\kappa))}\le\norm{\sigma_\kappa}_\infty.
\end{equation}
If $\sigma_\kappa=\sigma_\eta$, then $W(\kappa)-W(\eta)$ is compact.  Hence
\begin{equation}\label{eq:Lambda0}
  \Lambda_0(\sigma_\kappa)=\pi(W(\kappa)),\qquad \kappa\in\A,
\end{equation}
induces a contractive $*$-homomorphism on $\sigma(\A)$ and extends continuously to its uniform closure.
\end{proposition}

\begin{proof}
By Proposition \ref{prop:closure}, $\kappa^\vee\in\A$. The relation \eqref{eq:adjoint} and the product formula \eqref{eq:HSproduct}  imply that 
\[
 \pi(W(\kappa))^*\pi(W(\kappa))
 =\pi(W(\kappa^\vee\star\kappa))
 =\pi(W(\kappa\star\kappa^\vee))
 =\pi(W(\kappa))\pi(W(\kappa))^*.
\]
Consequently, the class is normal.

Let $x=\pi(W(\kappa))$.  For every $n \in \N_0 $, multiplicativity modulo compacts gives
$x^n=\pi(W(\kappa^{\star n}))$.  
Since \(x\) is normal in the Calkin \(C^*\)-algebra,
\(
  \|x\|=\rho(x),
\)
and the spectral radius formula yields
\[
\begin{aligned}
  \|x\|
  =\rho(x)
  &=
  \lim_{n\to\infty}\|x^n\|^{1/n}\le
  \lim_{n\to\infty}
  \|\kappa^{\star n}\|_{\Wiener}^{1/n}.
\end{aligned}
\]
By Proposition~\ref{prop:dictionary},
\[
  f_{\kappa^{\star n}}
  =
  f_\kappa^{*n},
  \qquad
  \|\kappa^{\star n}\|_{\Wiener}
  =
  \|f_\kappa^{*n}\|_{L^1(\R)}.
\]
Therefore
\[
  \|x\|
  \le
  \rho_{L^1(\R)}(f_\kappa),
\]
where
\[
  \rho_{L^1(\R)}(f_\kappa)
  =
  \lim_{n\to\infty}
  \|f_\kappa^{*n}\|_{L^1(\R)}^{1/n}
\]
is the spectral radius of \(f_\kappa\) in the convolution Banach
algebra \(L^1(\R)\).
The maximal ideal space of the group algebra $L^1(\R)$ is the dual group $\R$, and its Gelfand transform is the Fourier transform \cite{Rudin}.  Therefore, according to Proposition \ref{prop:dictionary},
\[
 r_{L^1(\R)}(f_\kappa)=\sup_{\xi\in\R}|\widehat f_\kappa(\xi)|
 =\norm{\sigma_\kappa}_\infty,
\]
which proves \eqref{eq:contractive}.  If $\sigma_\kappa=\sigma_\eta$, apply the estimate to $\kappa-\eta$.  Multiplicativity and preservation of the involution follow from Theorem \ref{thm:HSproduct} and Proposition \ref{prop:closure}.
\end{proof}

We next show directly that the Mellin symbols arising from $\A$ are
uniformly dense in $C_0(\R)$.

{\begin{lemma}[Density]\label{lem:density}
We have
\[
  \overline{\sigma(\mathcal A)}^{\|\cdot\|_\infty}
  =C_0(\mathbb R).
\]
More precisely, if
\(
  z(\xi)=\frac{1}{1/2+i\xi}
\),
then the non-unital \(C^*\)-algebra generated by \(z\) is
\(C_0(\mathbb R)\).
\end{lemma}}
\begin{proof}
Since $\A\subset\Wiener$, the Riemann--Lebesgue lemma gives $\sigma(\A)\subset C_0(\R)$.  On the other hand $\omega_C\in\A$, $\omega_C^\vee\in\A$, and $\A$ is a complex $\star$-algebra.  
According to \eqref{eq:symbol-C-Gamma},
\[
  z(\xi)=\sigma_{\omega_C}(\xi)
  =\frac{1}{1/2+i\xi}
\]
belongs to $\sigma(\A)$. Since $\A$ is closed under the involution
$\vee$ and
\(
  \sigma_{\kappa^\vee}=\overline{\sigma_\kappa}
\),
we also have $\overline z\in\sigma(\A)$. Moreover, multiplicative
convolution is transformed by the Mellin symbol into pointwise
multiplication. Hence $\sigma(\A)$ contains the nonunital
$*$-algebra $\mathcal B$ generated by $z$, that is, the finite
linear combinations of monomials
\(
  z^p\overline z^{\,q}\) for
  \(\qquad p+q\ge1
\).
Now $z\in C_0(\R)$, since $z(\xi)\to0$ as $|\xi|\to\infty$.
The algebra $\mathcal B$ is self-adjoint. Furthermore,
\eqref{eq:cesaro-symbol-geom} shows that $z$ is injective, so $\mathcal B$
separates points of $\R$. Finally, $z(\xi)\neq0$ for every
$\xi\in\R$, so $\mathcal B$ vanishes nowhere. The locally compact
Stone--Weierstrass theorem therefore gives
\[
  \overline{\mathcal B}^{\|\cdot\|_\infty}=C_0(\R).
\]
Since
\[
  \mathcal B\subset\sigma(\A)\subset C_0(\R),
\]
it follows that
\(
  \overline{\sigma(\A)}^{\|\cdot\|_\infty}=C_0(\R)
\).
\end{proof}

Combining Proposition \ref{prop:contractive} and Lemma \ref{lem:density}, we obtain a contractive $*$-homomorphism
\begin{equation}\label{eq:Lambda}
  \Lambda:C_0(\R)\longrightarrow  B(H^2)/\K,
  \qquad \Lambda(\sigma_\kappa)=\pi(W(\kappa)).
\end{equation}
The next theorem identifies it exactly.

\begin{theorem}[Calkin Mellin calculus]\label{thm:calkin-main}
The map \eqref{eq:Lambda} is an isometric $*$-isomorphism from $C_0(\R)$ onto the closed non-unital $C^*$-algebra
\begin{equation}\label{eq:range}
 \mathcal I_C
 =\{F(\pi(C)):F\in {C}(\SigmaC),\ F(0)=0\},
 \qquad \SigmaC=\{w:|w-1|=1\}.
\end{equation}
For every $\kappa\in\A$,
\begin{equation}\label{eq:functionalcalc}
  \pi(W(\kappa))=F_\kappa(\pi(C)),
\end{equation}
where $F_\kappa\in C(\SigmaC)$ is characterized by
\begin{equation}\label{eq:Fkappa}
 F_\kappa(0)=0,
 \qquad
 F_\kappa\!\left(\frac1{1/2+i\xi}\right)=\sigma_\kappa(\xi).
\end{equation}
Consequently,
\begin{equation}\label{eq:essresults}
  \norm{W(\kappa)}_{\mathrm{ess}}=\norm{\sigma_\kappa}_\infty,
  \qquad
  \spec_{\mathrm{ess}}(W(\kappa))=\ol{\sigma_\kappa(\R)}.
\end{equation}
\end{theorem}

\begin{proof}
By \cite[Theorem 3.3]{Young2004}  the essential spectrum of the Ces\`aro operator is
\begin{equation}\label{eq:essC}
  \spec_{\mathrm{ess}}(C)=\SigmaC.
\end{equation}
  By Proposition~\ref{prop:contractive}, $\pi(C)$ is normal.
Hence the continuous functional calculus
\cite[Chapter~4]{Douglas}
gives an isometric $*$-isomorphism
\[
  \Psi:I_{\Sigma_C}\longrightarrow\mathcal I_C,
  \qquad
  \Psi(F)=F(\pi(C))
\]  
where
\[
 I_{\SigmaC}=\{F\in C(\SigmaC):F(0)=0\},
 \, \Psi(F)=F(\pi(C)).
\]
By \eqref{eq:cesaro-symbol-geom},
\(
  z(\xi)=\frac{1}{1/2+i\xi}
\)
is a homeomorphism from $\R$ onto $\SigmaC\setminus\{0\}$, and
$z(\xi)\to0$ as $|\xi|\to\infty$. Hence every $g\in C_0(\R)$
determines a unique function $Ug\in C(\SigmaC)$ by
\[
  (Ug)(0)=0,
  \qquad
  (Ug)(z(\xi))=g(\xi).
\]
The condition $g(\xi)\to0$ at infinity guarantees continuity at the
point $0\in\SigmaC$. Moreover,
\(
  \|Ug\|_{C(\SigmaC)}=\|g\|_\infty
\)
and $U$ preserves products and complex conjugation. 
{Conversely, if \(F\in I_{\Sigma_C}\), define
\[
  g(\xi)=F(z(\xi)),\qquad \xi\in\mathbb R.
\]
Then \(g\) is continuous and, since \(z(\xi)\to0\) as
\(|\xi|\to\infty\),
\(
  g(\xi)\longrightarrow F(0)=0
\).
Thus \(g\in C_0(\mathbb R)\), and by construction \(Ug=F\).
}
Therefore
\[
  U:C_0(\mathbb R)\longrightarrow I_{\Sigma_C}
\]
is an isometric \( * \)-isomorphism.
We claim that
\[
  \Lambda=\Psi\circ U.
\]
Indeed, for the Ces\`aro symbol
\(
  z=\sigma_{\omega_C}
\),
we have
\(
  \Lambda(z)=\pi(W(\omega_C))=\pi(C)
\).
On the other hand, $Uz$ is the coordinate function on $\SigmaC$,
since
\[
  (Uz)(z(\xi))=z(\xi),
  \qquad
  (Uz)(0)=0.
\]
Therefore
\[
  (\Psi\circ U)(z)=\pi(C).
\]
Both $\Lambda$ and $\Psi\circ U$ are $*$-homomorphisms. By
Lemma~\ref{lem:density}, the nonunital $C^*$-algebra generated by
$z$ is all of $C_0(\R)$. Hence the two homomorphisms agree on
$C_0(\R)$.

Consequently, if
\(
  F_\kappa=U(\sigma_\kappa)
\),
so that
\[
  F_\kappa(0)=0,
  \qquad
  F_\kappa(z(\xi))=\sigma_\kappa(\xi),
\]
then
\[
  \pi(W(\kappa))
  =
  F_\kappa(\pi(C)).
\]
Since both $U$ and the continuous functional calculus are isometric,
\[
  \|W(\kappa)\|_{\rm ess}
  =
  \|\pi(W(\kappa))\|
  =
  \|\sigma_\kappa\|_\infty.
\]

{Finally, by the continuous spectral mapping theorem, \cite[Chapter 4]{Douglas}
\[
  \spec_{\mathrm{ess}}(W(\kappa))
  =
  F_\kappa(\Sigma_C).
\]
Now
\[
  F_\kappa(\Sigma_C\setminus\{0\})
  =
  \sigma_\kappa(\mathbb R),
  \qquad
  F_\kappa(0)=0.
\]
Since \(\sigma_\kappa\in C_0(\mathbb R)\),
\(
  0\in\overline{\sigma_\kappa(\mathbb R)}
\).
Therefore
\[
  F_\kappa(\Sigma_C)
  =
  \sigma_\kappa(\mathbb R)\cup\{0\}
  =
  \overline{\sigma_\kappa(\mathbb R)}
\]
and
\[
  \spec_{\mathrm{ess}}(W(\kappa))
  =
  \overline{\sigma_\kappa(\mathbb R)}.
\]
}

\end{proof}
\begin{remark}
We highlight that the identity
\[
  \Lambda(g)=(Ug)(\pi(C))
\]
holds for every \(g\in C_0(\R)\).  Notice, however, that for a general
\(g\in C_0(\R)\) the Calkin class \(\Lambda(g)\) is obtained by
continuity and need not be represented by a single sampled operator
\(W(\kappa)\) with \(\kappa\in\A\).
For \(g=\sigma_\kappa\), \(\kappa\in\A\), we have the concrete
identification
\[
  \pi(W(\kappa))
  =
  F_\kappa(\pi(C)),
  \qquad
  F_\kappa=U(\sigma_\kappa).
\]
\end{remark}

\begin{remark}[The closure in the essential spectrum]
The closure in \eqref{eq:essresults} is essential.  For example,
$\sigma_{\omega_C}(\R)=\SigmaC\setminus\{0\}$, whereas
$\spec_{\mathrm{ess}}(C)=\SigmaC$.
\end{remark}

\part{Toeplitz Algebra}
\section{The Toeplitz mechanism: sawtooth and Hilbert matrix}\label{sec:L3}
We now connect the discrete Mellin calculus with the Toeplitz algebra without using any external solution of the Ces\`aro problem.  The first ingredient is a reduction of polynomial Toeplitz defects to the Hilbert matrix, starting from the classical Toeplitz--Hankel semi-commutator identity \cite{BrownHalmos}.  We write
\[
  \Gamma_{jk}=\frac1{j+k+1},\qquad j,k\ge0.
\]
For $c\in L^\infty(\mathbb T)$ let $H_c=(I-P)M_c|_{H^2}$ denote the Hankel operator.

\begin{lemma}[Toeplitz semi-commutator]\label{lem:semicomm}
For $b,c\in L^\infty(\mathbb T)$,
\begin{equation}\label{eq:semicomm}
  T_bT_c-T_{bc}=-H_{\bar b}^*H_c.
\end{equation}
\end{lemma}

\begin{proof}
For $f\in H^2$, write $cf=P(cf)+(I-P)(cf)$.  Applying $PM_b$ gives
\[
 T_{bc}f=T_bT_cf+PM_b(I-P)M_cf.
\]
The last operator is $H_{\bar b}^*H_c$, since for $g\in(H^2)^\perp$ and $h\in H^2$,
\[
 \langle PM_bg,h\rangle=\langle g,\bar b h\rangle
 =\langle g,(I-P)(\bar b h)\rangle.
\]
Rearranging proves \eqref{eq:semicomm}.
\end{proof}
Let $b(e^{2\pi ix})=1-x$ for $0<x<1$, and define
\[
  A_m(n)=\widehat{b^m}(n)=\int_0^1(1-x)^m e^{-2\pi i n x}\,dx.
\]

\begin{lemma}[Common Hankel leading term]\label{lem:hankellead}
For every $m\ge1$ and $n\ne0$,
\begin{equation}\label{eq:Amrec}
 A_m(n)=\frac1{2\pi i n}-\frac{m}{2\pi i n}A_{m-1}(n)
       =\frac1{2\pi i n}+r_m(n),
 \, |r_m(n)|\le \frac{C_m}{n^2}.
\end{equation}
Consequently
\begin{equation}\label{eq:Hbm}
  H_{b^m}=-\frac1{2\pi i}\Gamma+R_m,
  \qquad R_m\in\HS,
\end{equation}
and $R_1=0$.
\end{lemma}

\begin{proof}
For \(n\neq0\), recall that
\[
 A_m(n)
 =
 \widehat{b^m}(n)
 =
 \frac{1}{2\pi}
 \int_0^{2\pi}
 \left(1-\frac{\theta}{2\pi}\right)^m
 e^{-in\theta}\,d\theta.
\]
Integrating by parts once, and using
\[
 b(0)=1,
 \qquad
 b(2\pi)=0,
 \qquad
 (b^m)'(\theta)
 =
 -\frac{m}{2\pi}b^{m-1}(\theta),
\]
gives
\begin{equation}\label{eq:Amrec2}
 A_m(n)
 =
 \frac{1}{2\pi i n}
 -
 \frac{m}{2\pi i n}A_{m-1}(n).
\end{equation}
This recurrence already contains the required asymptotic information.
Indeed, \(A_0(n)=0\) for \(n\neq0\), so
\[
 A_1(n)=\frac{1}{2\pi i n}.
\]
If \(m\ge2\), iterating \eqref{eq:Amrec} shows that the boundary term
\((2\pi i n)^{-1}\) remains the leading term, while every further term
contains at least two factors of \(n^{-1}\). Hence
\[
 A_m(n)
 =
 \frac{1}{2\pi i n}+r_m(n),
 \qquad
 |r_m(n)|\le \frac{C_m}{n^2},
 \qquad n\neq0.
\]

We now pass from the Fourier coefficients to the Hankel matrix.
Relative to the standard basis \(\{z^k\}_{k\ge0}\) of \(H^2\) and
\(\{z^{-j-1}\}_{j\ge0}\) of \((H^2)^\perp\),
\[
 (H_{b^m})_{jk}
 =
 \widehat{b^m}(-(j+k+1)).
\]
Applying the preceding estimate with
\[
 n=-(j+k+1)
\]
gives
\[
 (H_{b^m})_{jk}
 =
 -\frac{1}{2\pi i}\frac{1}{j+k+1}
 +
 r_m(-(j+k+1)).
\]
Thus
\[
 H_{b^m}
 =
 -\frac{1}{2\pi i}\Gamma+R_m,
\]
where
\[
 |(R_m)_{jk}|
 \le
 \frac{C_m}{(j+k+1)^2}.
\]
Consequently,
\[
 \|R_m\|_{\HS}^2
 \le
 C_m^2
 \sum_{j,k\ge0}\frac{1}{(j+k+1)^4}.
\]
Grouping together the pairs with \(j+k+1=N\), of which there are
exactly \(N\), we obtain
\[
 \sum_{j,k\ge0}\frac{1}{(j+k+1)^4}
 =
 \sum_{N\ge1}\frac{N}{N^4}
 =
 \sum_{N\ge1}\frac1{N^3}
 <\infty.
\]
Hence \(R_m\in\HS\).

For \(m=1\), \eqref{eq:Amrec} and \(A_0(n)=0\) give the exact formula
\(
 A_1(n)={1}/{2\pi i n}
\),
so in this case \(R_1=0\).
\end{proof}

\begin{theorem}[Polynomial Toeplitz defect]\label{thm:L3}
Let $p\in\mathbb C[v]$ satisfy $p(0)=p(1)=0$, and write
\begin{equation}\label{eq:qdef}
  p(v)=v(v-1)q(v).
\end{equation}
Then
\begin{equation}\label{eq:L3}
  p(T_b)-T_{p\circ b}
  =-\frac1{4\pi^2}q(T_b)\Gamma^2+K_p,
  \qquad K_p\in\HS.
\end{equation}
For $p(v)=v-v^2$ the remainder vanishes and
\(
  T_{b^2}-T_b^2=\frac1{4\pi^2}\Gamma^2
\).
\end{theorem}

\begin{proof}
Put $D_n=T_b^n-T_{b^n}$.  The telescoping identity
\[
 D_n=-\sum_{m=1}^{n-1}T_b^{n-m-1}H_{\bar b}^*H_{b^m}
\]
follows from Lemma \ref{lem:semicomm}.  Since $b$ is real-valued, Lemma \ref{lem:hankellead} gives
\[
 -H_{\bar b}^*H_{b^m}
 =-\frac1{4\pi^2}\Gamma^2+E_m,
 \qquad E_m\in\HS,
\]
and hence
\begin{equation}\label{eq:Dn}
 D_n=-\frac1{4\pi^2}\left(\sum_{k=0}^{n-2}T_b^k\right)\Gamma^2+K_n,
 \qquad K_n\in\HS.
\end{equation}
If
\(
  p(v)=\sum_{n=1}^d c_n v^n
\),
then
\(
  p(T_b)-T_{p\circ b}
  =
  \sum_{n=2}^d c_nD_n
\).
Substituting \eqref{eq:Dn} gives
\[
  p(T_b)-T_{p\circ b}
  =
  -\frac{1}{4\pi^2}s(T_b)\Gamma^2+K_p,
  \qquad K_p\in\HS,
\]
where
\[
  s(v)
  =
  \sum_{n=2}^d c_n\sum_{k=0}^{n-2}v^k.
\]
It remains to identify \(s\). Since
\(
  (v-1)\sum_{k=0}^{n-2}v^k=v^{n-1}-1
\),
we have
\[
\begin{aligned}
  (v-1)s(v)
  &=
  \sum_{n=2}^d c_n(v^{n-1}-1)=
  \sum_{n=2}^d c_nv^{n-1}
  -
  \sum_{n=2}^d c_n.
\end{aligned}
\]
Now \(p(1)=0\) implies
\(
  -\sum_{n=2}^d c_n=c_1
\),
while \(p(0)=0\) gives
\(
  \frac{p(v)}v
  =
  c_1+\sum_{n=2}^d c_nv^{n-1}
\).
Hence
\(
  (v-1)s(v)=\frac{p(v)}v
\).
Thus
\[
  p(T_b)-T_{p\circ b}
  =
  -\frac{1}{4\pi^2}q(T_b)\Gamma^2+K_p.
\]
\end{proof}

The relation with the Mellin symbol is especially simple.  Since the shifted Hilbert matrix $W(\omega_\Gamma)$ differs from $\Gamma$ by a Hilbert--Schmidt operator and Theorem \ref{thm:HSproduct} applies to $(\omega_\Gamma,\omega_\Gamma)$,
\begin{equation}\label{eq:Gamma2symbol}
 \pi(\Gamma^2)=\pi(W(\omega_\Gamma^{\star2})),
 \qquad
 \frac{\sigma_{\omega_\Gamma^{\star2}}(\xi)}{4\pi^2}
 =\frac14\operatorname{sech}^2(\pi\xi)
 =\thetaM(\xi)(1-\thetaM(\xi)),
\end{equation}
where
\begin{equation}\label{eq:theta}
  \thetaM(\xi)=\frac12-\frac12\tanh(\pi\xi)=\frac1{1+e^{2\pi\xi}}.
\end{equation}
\section{The principal-value orbit: stability, density, and intertwining}
\label{sec:poleorbit}
The second internal ingredient is the action of the singular kernel
\[
  \omega_b(u)=\frac1{2\pi i(1-u)},\qquad u\ne1,
  \qquad \omega_b(1)=0.
\]
It does not belong to $\Wiener$, so the ordinary $L^1$ Mellin convolution theorem does not apply.  The singularity is handled in principal value.

We begin by introducing a smooth logarithmic class on which the principal-value convolution with the pole kernel is well defined and preserves the weighted regularity needed later.
For $m\in\N_0$ write $\kappa\in\Finf_m$ if $f_\kappa\in C^\infty(\mathbb R)$ and for every $j\ge0$ there is $K_j$ such that
\begin{equation}\label{eq:Finf}
  |f_\kappa^{(j)}(t)|\le K_j\mu_m(t).
\end{equation}
Then $\Finf_m\subset\FBV_m$ (take $J=\varnothing$ and use the bounds for derivative orders $0$ and $1$).

Define the Euler derivative $\Theta=v\,d/dv$ and
\[
 w_m(v)=(1+|\log v|)^m\min(1,v^{-1}).
\]
Since $e^{t/2}w_m(e^t)=(1+|t|)^m e^{t/2}\min(1,e^{-t})
 =(1+|t|)^m e^{-|t|/2}=\mu_m(t)$, the chain rule gives the following useful equivalence.

\begin{lemma}[Logarithmic and Euler derivatives]\label{lem:euler}
A kernel $\kappa$ belongs to $\Finf_m$ if and only if $\kappa\in C^\infty(0,\infty)$ and, for every $j\ge0$,
\[
 |\Theta^j\kappa(v)|\le L_jw_m(v).
\]
If \eqref{eq:Finf} holds for $j=0,1,2$, then
\begin{equation}\label{eq:KmfromF}
 |\kappa(v)|+|v\kappa'(v)|+|v^2\kappa''(v)|
 \le \left(K_2+3K_1+\frac94K_0\right)w_m(v).
\end{equation}
\end{lemma}
\begin{proof}
Put
\(
  g(t)=\kappa(e^t)
\).
Since \(v=e^t\), the chain rule gives
\[
  g'(t)=e^t\kappa'(e^t)
       =(\Theta\kappa)(e^t),
\]
and, by iteration,
\[
  g^{(j)}(t)=(\Theta^j\kappa)(e^t),
  \qquad j\ge0.
\]
On the other hand,
\(
  f_\kappa(t)=e^{t/2}g(t)
\).
Therefore repeated differentiation yields
\[
  f_\kappa^{(j)}(t)
  =
  e^{t/2}\left(\partial_t+\frac12\right)^j g(t).
\]
Thus bounds for the Euler derivatives
\((\Theta^\ell\kappa)(e^t)=g^{(\ell)}(t)\), \(0\le \ell\le j\),
give the corresponding bounds for \(f_\kappa^{(j)}\).
Conversely,
\[
  g(t)=e^{-t/2}f_\kappa(t),
\]
and therefore
\[
  g^{(j)}(t)
  =
  e^{-t/2}
  \left(\partial_t-\frac12\right)^j f_\kappa(t).
\]
Since \(g^{(j)}(t)=(\Theta^j\kappa)(e^t)\), we obtain
\[
  (\Theta^j\kappa)(e^t)
  =
  e^{-t/2}
  \left(\partial_t-\frac12\right)^j f_\kappa(t).
\]
Using
\[
  e^{t/2}w_m(e^t)=\mu_m(t),
  \qquad
  e^{-t/2}\mu_m(t)=w_m(e^t),
\]
the two families of estimates are therefore equivalent.
For the quantitative estimate, the case \(j=0\) gives
\[
  \kappa(e^t)=e^{-t/2}f_\kappa(t).
\]
For \(j=1\),
\[
  v\kappa'(v)
  =
  (\Theta\kappa)(v)
  =
  e^{-t/2}
  \left(
    f_\kappa'(t)-\frac12f_\kappa(t)
  \right).
\]
For the second derivative, note that
\(
  \Theta^2\kappa
  =
  v\kappa'+v^2\kappa''
\),
hence
\(
  v^2\kappa''
  =
  \Theta^2\kappa-\Theta\kappa
\).
Since
\[
  \Theta^2\kappa(e^t)
  =
  e^{-t/2}
  \left(
    f_\kappa''-f_\kappa'
    +\frac14f_\kappa
  \right),
\]
we obtain
\[
  v^2\kappa''
  =
  e^{-t/2}
  \left(
    f_\kappa''-2f_\kappa'
    +\frac34f_\kappa
  \right).
\]
If \eqref{eq:Finf} holds for \(j=0,1,2\), then
\[
\begin{aligned}
 |\kappa(v)|
 &\le K_0w_m(v),\\
 |v\kappa'(v)|
 &\le
 \left(K_1+\frac12K_0\right)w_m(v),\\
 |v^2\kappa''(v)|
 &\le
 \left(K_2+2K_1+\frac34K_0\right)w_m(v).
\end{aligned}
\]
Adding these inequalities gives
\[
  |\kappa(v)|+|v\kappa'(v)|+|v^2\kappa''(v)|
  \le
  \left(
    K_2+3K_1+\frac94K_0
  \right)w_m(v),
\]
which is \eqref{eq:KmfromF}.
\end{proof}

The pole profile is
\begin{equation}\label{eq:poleprofile}
 p(t)=f_{\omega_b}(t)=\frac{i}{4\pi\sinh(t/2)},\quad t\ne0,
 \qquad p(0)=0.
\end{equation}
It is odd, satisfies $p(t)=i/(2\pi t)+O(t)$ as $t\to0$, and decays like $e^{-|t|/2}$ at infinity.  Moreover
\begin{equation}\label{eq:poleFT}
 \PV\int_{\mathbb R}p(t)e^{i\xi t}\,dt=-\frac12\tanh(\pi\xi)=\sigma_{\omega_b}(\xi).
\end{equation}
{
Since the kernel \(\omega_b\) has a pole at \(u=1\), the integral
defining multiplicative convolution in \eqref{eq:mult-conv} is not
absolutely convergent when one of the factors is \(\omega_b\).
Accordingly, for a sufficiently regular kernel \(\kappa\), we interpret
\(\omega_b\star\kappa\) in the logarithmically symmetric principal-value
sense
\begin{equation}\label{eq:pvstar}
  (\omega_b\star\kappa)(w)
  =
  \lim_{\varepsilon\downarrow0}
  \int_{(0,\infty)\setminus(e^{-\varepsilon},e^\varepsilon)}
  \omega_b(u)\,
  \kappa\!\left(\frac wu\right)\frac{du}{u},
\end{equation}
whenever the limit exists.  Under the change of variables \(u=e^t\),
this is precisely the symmetric cutoff \(|t|>\varepsilon\) in
logarithmic coordinates.  The proof below also shows that, on the
class considered here, the same value is obtained from an additively
symmetric excision about \(u=1\).
}

\begin{theorem}[Stability under the pole convolution]\label{thm:pole-invariance}
If \(\kappa\in\Finf_m\), then the principal-value convolution
\(\omega_b\star\kappa\), understood in the sense of
\eqref{eq:pvstar}, exists pointwise and
\(
  \omega_b\star\kappa\in\Finf_{m+1}
\).
Furthermore,
\begin{equation}\label{eq:polemult}
  \sigma_{\omega_b\star\kappa}=\sigma_{\omega_b}\sigma_\kappa,
  \qquad
  \sigma_{\omega_b\star\kappa+\frac12\kappa}=\thetaM\,\sigma_\kappa.
\end{equation}
\end{theorem}
\begin{proof}
Let \(g=f_\kappa\), \(p=f_{\omega_b}\).
For each $j\ge0$, set
\[
  N_j(t)
  =
  \int_0^1
  \bigl(
    g^{(j)}(t-r)-g^{(j)}(t+r)
  \bigr)p(r)\,dr,
\]
and
\[
  F_j(t)
  =
  \int_{|r|>1}
  g^{(j)}(t-r)p(r)\,dr.
\]
We write
\(
  H_j:=N_j+F_j
\).

We first verify that these quantities are well defined. By the
fundamental theorem of calculus,
\[
  |g^{(j)}(t-r)-g^{(j)}(t+r)|
  \le
  2r\sup_{|s-t|\le1}|g^{(j+1)}(s)|,
  \qquad 0<r\le1.
\]
Since $r|p(r)|$ is bounded near the origin, the integral defining
$N_j(t)$ converges absolutely. The far integral also converges
absolutely, because
\[
  |p(r)|\lesssim e^{-|r|/2},
  \qquad |r|>1,
\]
while $g^{(j)}$ satisfies the defining weighted estimate of
$\Finf_m$.

We now identify $H_j$ with the logarithmically symmetric principal
value. For $0<\varepsilon<1$, substituting $r\mapsto-r$ on
$(-1,-\varepsilon)$ and using $p(-r)=-p(r)$ gives
\[
\begin{aligned}
  \int_{\varepsilon<|r|\le1}
  g^{(j)}(t-r)p(r)\,dr
  &=
  \int_\varepsilon^1
  \bigl(
    g^{(j)}(t-r)-g^{(j)}(t+r)
  \bigr)p(r)\,dr\\
  &\longrightarrow
  N_j(t),
  \qquad \varepsilon\to 0.
\end{aligned}
\]
Together with the absolutely convergent far part, this yields
\[
  H_j(t)
  =
  \operatorname{p.v.}
  \int_{\mathbb R}
  g^{(j)}(t-r)p(r)\,dr,
\]
where the principal value is taken with the symmetric cutoff
$|r|>\varepsilon$.

We next identify $H_0$ with the multiplicative principal value
\eqref{eq:pvstar}. Put $w=e^t$. With $u=e^r$,
\[
\begin{aligned}
  \int_{|r|>\varepsilon}
  p(r)g(t-r)\,dr
  &=
  e^{t/2}
  \int_{(0,e^{-\varepsilon})\cup(e^\varepsilon,\infty)}
  \omega_b(u)
  \kappa\!\left(\frac wu\right)\frac{du}{u}.
\end{aligned}
\]
Near $u=1$, write
\[
  \omega_b(u)
  \kappa\!\left(\frac wu\right)\frac1u
  =
  \frac{q_w(u)}{1-u},
  \qquad
  q_w(u)
  :=
  \frac{\kappa(w/u)}{2\pi i\,u}.
\]
Since $\kappa\in C^\infty(0,\infty)$, the function $q_w$ is
Lipschitz in a neighborhood of $u=1$. Decompose
\[
  \frac{q_w(u)}{1-u}
  =
  \frac{q_w(1)}{1-u}
  +
  \frac{q_w(u)-q_w(1)}{1-u}.
\]
The second term is locally bounded, so its limit is independent of the
particular shrinking window. It remains only to compare the two
cutoffs for the singular constant term.

Fix $0<\rho<1$, and set
\[
  \alpha_\varepsilon
  =
  1-e^{-\varepsilon},
  \qquad
  \beta_\varepsilon
  =
  e^\varepsilon-1.
\]
For $\varepsilon$ sufficiently small,
$\alpha_\varepsilon,\beta_\varepsilon<\rho$, and
\[
\begin{aligned}
 &\int_{1-\rho}^{e^{-\varepsilon}}\frac{du}{1-u}
  +
  \int_{e^\varepsilon}^{1+\rho}\frac{du}{1-u}=
  \log\frac{\beta_\varepsilon}{\alpha_\varepsilon}
  =
  \varepsilon.
\end{aligned}
\]
By contrast, for every $0<\eta<\rho$,
\[
  \int_{1-\rho}^{1-\eta}\frac{du}{1-u}
  +
  \int_{1+\eta}^{1+\rho}\frac{du}{1-u}
  =
  0.
\]
Thus the discrepancy between the logarithmically symmetric and the
additively symmetric truncations of the singular term is exactly
$\varepsilon$, and therefore tends to zero. Consequently the two
principal-value prescriptions agree, and
\[
  H_0(t)
  =
  f_{\omega_b\star\kappa}(t).
\]

We next justify differentiation under the integral signs. Fix a
compact interval $I\subset\mathbb R$ and let
\[
  I_1
  :=
  \{s\in\mathbb R:\operatorname{dist}(s,I)\le1\}.
\]
For the near part,
\[
  |g^{(j+1)}(t-r)-g^{(j+1)}(t+r)|
  \le
  2r\sup_{s\in I_1}|g^{(j+2)}(s)|,
\]
for $t\in I$ and $0<r\le1$. Since $r|p(r)|$ is bounded near the
origin, this provides an integrable majorant independent of
$t\in I$. Hence
\[
  N_j'=N_{j+1}.
\]
For the far part,
\[
  |g^{(j+1)}(t-r)p(r)|
  \le
  C_{I,j}(1+|r|)^m e^{-|r|},
  \qquad |r|>1,\quad t\in I,
\]
which is integrable in $r$. Therefore
\[
  F_j'=F_{j+1}.
\]
It follows that
\[
  H_j'=H_{j+1},
\]
and hence, by induction,
\[
  H_0^{(j)}=H_j,
  \qquad j\ge0.
\]

It remains to establish the weighted estimates. If $|s-t|\le1$, then
\[
  \mu_m(s)\le C_m\mu_m(t),
\]
and therefore
\[
  |N_j(t)|
  \le
  C_mK_{j+1}\mu_m(t).
\]
For the far part, enlarging the domain of integration and using the
decay of $p$ for $|r|>1$, we obtain
\[
\begin{aligned}
  |F_j(t)|
  &\le
  CK_j
  \int_{\mathbb R}
  (1+|t-r|)^m
  e^{-(|t-r|+|r|)/2}\,dr.
\end{aligned}
\]
Let $J_t$ denote the closed segment joining $0$ and $t$. Since
\[
  |t-r|+|r|
  =
  |t|+2\operatorname{dist}(r,J_t),
\]
the contribution of $J_t$ is
\[
  O\!\left(
    (1+|t|)^{m+1}e^{-|t|/2}
  \right),
\]
while the two complementary rays contribute
\[
  O\!\left(
    (1+|t|)^m e^{-|t|/2}
  \right).
\]
Thus
\[
  |F_j(t)|
  \le
  C_{m,j}K_j\mu_{m+1}(t).
\]
Since $\mu_m\le\mu_{m+1}$, we conclude that
\[
  |H_0^{(j)}(t)|
  =
  |H_j(t)|
  \le
  C_{m,j}(K_j+K_{j+1})\mu_{m+1}(t),
  \qquad j\ge0.
\]
Hence
\[
  \omega_b\star\kappa\in\Finf_{m+1}.
\]

It remains to identify the Mellin symbol. For $\varepsilon>0$, set
\[
  p_\varepsilon
  :=
  p\,\mathbf 1_{\{|r|>\varepsilon\}},
  \qquad
  h_\varepsilon
  =
  p_\varepsilon*g.
\]
Then $p_\varepsilon,g\in L^1(\mathbb R)$ and
\[
  \widehat h_\varepsilon
  =
  \widehat p_\varepsilon\,\widehat g.
\]
For $0<\delta<\varepsilon<1$, oddness of $p$ and the fundamental
theorem of calculus give
\[
\begin{aligned}
  \|h_\delta-h_\varepsilon\|_1
  &\le
  \int_\delta^\varepsilon
  |p(r)|
  \,
  \|g(\cdot-r)-g(\cdot+r)\|_1\,dr\le
  \frac{\|g'\|_1}{\pi}
  (\varepsilon-\delta).
\end{aligned}
\]
Thus $(h_\varepsilon)$ is Cauchy in $L^1(\mathbb R)$ and converges to
some $h\in L^1(\mathbb R)$. On the other hand, by the principal-value
construction,
\[
  h_\varepsilon(t)\longrightarrow H_0(t)
\]
for every $t\in\mathbb R$. Choosing an $L^1$-convergent subsequence
which converges almost everywhere, we obtain
\[
  h=H_0
\]
almost everywhere. Hence
\[
  \|h_\varepsilon-H_0\|_1\longrightarrow0,
\]
and consequently
\[
  \widehat h_\varepsilon
  \longrightarrow
  \widehat H_0
\]
uniformly on $\mathbb R$.

Moreover, since $p$ is odd,
\[
  \widehat p_\varepsilon(\xi)
  =
  2i\int_\varepsilon^\infty
  p(r)\sin(\xi r)\,dr.
\]
For each fixed $\xi$, the function
$p(r)\sin(\xi r)$ is integrable on $(0,\infty)$: near the origin this
follows from $p(r)=O(r^{-1})$ and $\sin(\xi r)=O(r)$, while at
infinity it follows from the exponential decay of $p$. Therefore,
for each fixed $\xi$, dominated convergence gives
\[
  \widehat p_\varepsilon(\xi)
  \longrightarrow
  \operatorname{p.v.}
  \int_{\mathbb R}p(r)e^{i\xi r}\,dr
  =
  \sigma_{\omega_b}(\xi).
\]
Since
\[
  \widehat h_\varepsilon
  =
  \widehat p_\varepsilon\,\widehat g,
  \qquad
  \widehat g=\sigma_\kappa,
\]
passing to the limit for each fixed $\xi$ yields
\[
  \sigma_{\omega_b\star\kappa}(\xi)
  =
  \widehat H_0(\xi)
  =
  \sigma_{\omega_b}(\xi)\sigma_\kappa(\xi).
\]
Finally,
\[
  \sigma_{\omega_b}(\xi)+\frac12
  =
  \widetilde\vartheta(\xi),
\]
so
\[
  \sigma_{\omega_b\star\kappa+\frac12\kappa}
  =
  \widetilde\vartheta\,\sigma_\kappa.
\]
This proves \eqref{eq:polemult}.
\end{proof}

\begin{remark}[Sharpness of the weight loss]\label{rem:polesharp}
The passage from \(\Finf_m\) to \(\Finf_{m+1}\) in
Theorem~\ref{thm:pole-invariance} is, in general, unavoidable.
Indeed, for \(\kappa=\omega_\Gamma\in\Finf_0\),
\[
\begin{aligned}
  \sigma_{\omega_b}(\xi)\sigma_{\omega_\Gamma}(\xi)
  &=
  -\frac{\pi}{2}
  \tanh(\pi\xi)\text{sech}(\pi\xi)=
  \frac12\frac{d}{d\xi}\text{sech}(\pi\xi).
\end{aligned}
\]
With our Fourier convention, this gives
\[
  f_{\omega_b\star\omega_\Gamma}(t)
  =
  \frac{it}{4\pi\cosh(t/2)}.
\]
Equivalently,
\[
  (\omega_b\star\omega_\Gamma)(v)
  =
  \frac{i\log v}{2\pi(1+v)}.
\]
Hence
\[
  \omega_b\star\omega_\Gamma
  \in
  \Finf_1\setminus\Finf_0.
\]
Thus the loss of one polynomial weight inin Theorem~\ref{thm:pole-invariance} cannot in general be removed.
\end{remark}
The preceding theorem shows that the operation
\[
  \kappa\longmapsto
  \omega_b\star\kappa+\frac12\kappa
\]
preserves the smooth weighted classes and, on Mellin symbols,
corresponds to multiplication by \(\thetaM\).  We now iterate this
operation from a suitable initial kernel.  The resulting orbit remains
regular, while its Mellin symbols generate a dense subalgebra of
\(C_0(\mathbb R)\).  This will provide the class of kernels used below
to connect the Mellin calculus with the Toeplitz algebra.  
Set
\[
 \kappa_0=\omega_\Gamma^{\star2},
 \qquad
 \kappa_{r+1}=\omega_b\star\kappa_r+\frac12\kappa_r,
 \qquad r\ge0,
\]
and let $D$ be the complex $\star$-algebra generated by
\[
  \omega_C,\quad \omega_C^\vee,\quad \omega_\Gamma,
  \quad\text{and}\quad \{\kappa_r:r\ge0\}.
\]
The pole kernel itself is not an element of $D$; it enters only through the principal-value recursion.

\begin{proposition}[Closure of the domain]\label{prop:Dclosure}
Every element of $D$ belongs to $\A$.  More precisely, $\kappa_r\in\Finf_{r+1}$, finite $\star$-products of generators belong to some $\FBV_m$, and finite linear combinations remain in some $\FBV_M$.  Moreover $D$ is closed under $\vee$, every $\kappa\in D$ belongs to $\Wiener$ and satisfies $(R)$, and the product defect is Hilbert--Schmidt for every ordered pair in $D$.
\end{proposition}

\begin{proof}
The Hilbert kernel satisfies
\(
  f_{\omega_\Gamma}(t)=\frac12\operatorname{sech}(t/2)
\),
and each derivative of this function is bounded by a constant
multiple of \(\mu_0\). Hence \(\omega_\Gamma\in\Finf_0\).
Since
\(
  f_{\kappa_0}
  =
  f_{\omega_\Gamma}*f_{\omega_\Gamma}
\),
we have, for every \(j\ge0\),
\(
  (f_{\kappa_0})^{(j)}
  =
  f_{\omega_\Gamma}^{(j)}*f_{\omega_\Gamma}
\).
Thus
\(
  |(f_{\kappa_0})^{(j)}|
  \lesssim
  \mu_0*\mu_0
  \lesssim
  \mu_1
\),
and therefore \(\kappa_0\in\Finf_1\).

Suppose inductively that \(\kappa_r\in\Finf_{r+1}\).
By Theorem~\ref{thm:pole-invariance},
\(
  \omega_b\star\kappa_r\in\Finf_{r+2}
\).
Since \(\mu_{r+1}\le\mu_{r+2}\), we also have
\(\kappa_r\in\Finf_{r+2}\). Hence
\[
  \kappa_{r+1}
  =
  \omega_b\star\kappa_r+\frac12\kappa_r
  \in\Finf_{r+2}.
\]
Consequently,
\[
  \kappa_r\in\Finf_{r+1},
  \qquad r\ge0.
\]
Since \(\Finf_m\subset\FBV_m\), every generator of \(D\) belongs to some weighted \(BV\) class.  Proposition~\ref{prop:closure}
gives
\(
  \FBV_m\star\FBV_n
  \subset
  \FBV_{m+n+1}
\).
Hence every finite \(\star\)-product of generators belongs to
\(\FBV_M\) for some \(M\ge0\).

The pointwise variation measure is subadditive, so finite linear combinations remain in a weighted $BV$ class.  Ordinary $\star$-products are compatible with $\vee$ by Proposition \ref{prop:dictionary}, and $\omega_C^\vee$ is already a generator while $\omega_\Gamma^\vee=\omega_\Gamma$.

It remains to check the principal-value orbit pointwise, because the Fourier multiplier identity alone would identify only almost-everywhere classes.  Put $p=f_{\omega_b}$ and $g=f_\kappa$ for $\kappa\in\Finf_m$.  The pole profile is odd and purely imaginary, hence
$\overline{p(-s)}=p(s)$.  For the symmetric truncations
\[
 H_\varepsilon(t)=\int_{|s|>\varepsilon}p(s)g(t-s)\,ds
\]
one has, after $s\mapsto-s$,
\[
 \overline{H_\varepsilon(-t)}
 =\int_{|s|>\varepsilon}p(s)\,\overline{g(-t+s)}\,ds.
\]
Letting $\varepsilon\downarrow0$ and using Theorem \ref{thm:pole-invariance} for both $\kappa$ and $\kappa^\vee$ gives the exact identity
\begin{equation}\label{eq:pvvee}
 (\omega_b\star\kappa)^\vee=\omega_b\star\kappa^\vee.
\end{equation}
Now $\kappa_0=\omega_\Gamma^{\star2}$ is $\vee$-fixed, and the recursion
$\kappa_{r+1}=\omega_b\star\kappa_r+\frac12\kappa_r$ together with \eqref{eq:pvvee} shows inductively that every $\kappa_r$ is $\vee$-fixed.  Hence the generating set, and therefore $D$, is $\vee$-invariant.  
Finally, since \(D\subset\A\), Lemma~\ref{lem:regularity} shows
that every \(\kappa\in D\) belongs to \(\Wiener\) and satisfies
\textbf{(R)}.  Moreover, for every ordered pair
\(\kappa,\eta\in D\), Theorem~\ref{thm:HSproduct} gives
\[
  W(\kappa)W(\eta)-W(\kappa\star\eta)\in\mathcal S_2.
\]\end{proof}

The orbit has been chosen so that its symbols are exactly the polynomial functions of the logistic coordinate that vanish at the two endpoints.

\begin{proposition}[Orbit symbols and density]\label{prop:orbitdensity}
For $r\ge0$,
\begin{equation}\label{eq:orbitsymbol}
  \sigma_{\kappa_r}(\xi)
  =4\pi^2\thetaM(\xi)^{r+1}\bigl(1-\thetaM(\xi)\bigr).
\end{equation}
Consequently
\begin{equation}\label{eq:Idef}
 I=\{p\circ\thetaM:p\in\mathbb C[v],\ p(0)=p(1)=0\}
 \subset \sigma(D),
\end{equation}
and $I$ is dense in $C_0(\mathbb R)$.  In particular,
\begin{equation}\label{eq:sigmaD}
  \overline{\sigma(D)}^{\|\cdot\|_\infty}=C_0(\mathbb R).
\end{equation}
\end{proposition}

\begin{proof}
By \eqref{eq:Gamma2symbol},
\[
 \sigma_{\kappa_0}=\pi^2\operatorname{sech}^2(\pi\xi)
 =4\pi^2\thetaM(1-\thetaM).
\]
The recursion and \eqref{eq:polemult} give $\sigma_{\kappa_{r+1}}=\thetaM\sigma_{\kappa_r}$, proving \eqref{eq:orbitsymbol}.  If $p(0)=p(1)=0$, write $p(v)=v(v-1)q(v)$ with $q(v)=\sum_{r=0}^Nq_rv^r$.  Then
\[
  \kappa_p=-\frac1{4\pi^2}\sum_{r=0}^Nq_r\kappa_r\in D,
  \qquad \sigma_{\kappa_p}=p\circ\thetaM,
\]
which proves \eqref{eq:Idef}.

Since $\thetaM$ is a homeomorphism from $\mathbb R$ onto $(0,1)$, composition with $\thetaM$ identifies $C_0((0,1))$ isometrically with $C_0(\mathbb R)$.  Polynomials vanishing at $0$ and $1$ form a self-adjoint subalgebra of $C_0((0,1))$ which separates points and vanishes nowhere; the locally compact Stone--Weierstrass theorem gives density.
\end{proof}
\begin{remark}
    The endpoint conditions are natural here.  Indeed,
\[
  \thetaM(\xi)\to0\quad(\xi\to+\infty),
  \qquad
  \thetaM(\xi)\to1\quad(\xi\to-\infty),
\]
so \(p\circ\thetaM\) belongs to \(C_0(\mathbb R)\) precisely when
\(p(0)=p(1)=0\).  Moreover, these are exactly the polynomial
functions generated by the orbit symbols, since
\[
  \sigma_{\kappa_r}
  =
  4\pi^2\thetaM^{r+1}(1-\thetaM).
\]
\end{remark}

\section{Principal-value intertwining}\label{sec:intertwining}

Proposition~\ref{prop:orbitdensity} identifies the symbols reachable
by the orbit $\mathcal D$, but says nothing yet about $T_b$ itself: to
use Theorem~\ref{thm:L3} 
we must identify the action of
\(\pi(T_b)\) on the Mellin fibre, by which we mean the Calkin
subalgebra
\(
  \Lambda(C_0(\mathbb R))
\).
More precisely, we shall show that left multiplication by
\(\pi(T_b)\) corresponds, under the Mellin symbol calculus, to
multiplication by \(\widetilde\theta\). The difficulty is that $\omega_b$ has a
pole at $1$, so Theorem~\ref{thm:HSproduct} does not apply directly to
the pair $(\omega_b,\kappa)$; this section supplies the quantitative
replacement.

Fix $j,l\ge0$ and put $a=j+1$, $e=l+1$, $\beta=e/a$. For a sufficiently
regular kernel $\kappa$, define
\begin{equation}\label{eq:phiGQ}
  \phi(t) = \frac{a}{2\pi i}\, \frac{\kappa(e/t)}{t}, \qquad
  G(t) = \frac{\phi(t)}{a-t}, \qquad
  Q(t) = \frac{\phi(t)-\phi(a)}{a-t}.
\end{equation}
Then $G = Q + \phi(a)/(a-t)$: the singularity has been split into a
regular divided difference and a universal pole. With $H_n$ we denote the $n-$th Harmonic number and with $\gamma$ the Euler constant.

\begin{proposition}[Principal-value pairing]\label{prop:pvpairing}
Assume that $W(\kappa)$ is bounded, that $G$ is integrable away from
$a$, and that $Q$ is locally integrable near $a$. With
\[
  \Delta_Q(K) = \sum_{\substack{1\le c\le K \\ c\ne a}} Q(c) - \int_0^K Q(t)\,dt,
\]
interpreted with the common cutoff $K$, one has
\begin{equation}\label{eq:pvpairing}
  \bigl[ W(\omega_b)W(\kappa) - W(\omega_b\star\kappa) \bigr]_{jl}
  = \frac1a \Bigl[ \lim_{K\to\infty} \Delta_Q(K) + \phi(a)\bigl(\psi(a)-\log a\bigr) \Bigr].
\end{equation}
\end{proposition}
\begin{proof}
For \(c=k+1\), the definition of \(W(\omega_b)\), together with the convention \(\omega_b(1)=0\), gives
\[
  W(\omega_b)_{jk}
  =
  \begin{cases}
    \displaystyle \frac{1}{2\pi i(a-c)}, & c\neq a,\\[2mm]
    0, & c=a.
  \end{cases}
\]
Moreover,
\(
  W(\kappa)_{kl}
  =
  \frac1c\,\kappa\!\left(\frac ec\right)
\) and
\[
 [W(\omega_b)W(\kappa)]_{jl}
 =
 \sum_{c\ge1}'
 \frac{\kappa(e/c)}
      {2\pi i\,c(a-c)}=
 \frac1a\sum_{c\ge1}'G(c),
\]
where the prime indicates omission of the term \(c=a\).
Since \(W(\omega_b)=T_b-\frac12I\) and \(W(\kappa)\) are bounded,
the \(j\)-th row of \(W(\omega_b)\) and the \(l\)-th column of
\(W(\kappa)\) belong to \(\ell^2\).  Thus the last series converges
absolutely by Cauchy--Schwarz.

We next consider the convolution term.  By hypothesis, \(G\) is
integrable away from \(a\).  Near \(a\) we have
\(
  G(t)
  =
  Q(t)+\frac{\phi(a)}{a-t}
\).
The first term belongs to \(L^1_{\mathrm{loc}}\) near \(a\), while
the second has a symmetric principal value there.  Consequently
\(
  \PV\int_0^\infty G(t)\,dt
\)
exists.

Using the definition of the principal-value convolution and making
the change of variables \(t=au\), we obtain
\[
\begin{aligned}
 [W(\omega_b\star\kappa)]_{jl}
 &=
 \frac1a
 (\omega_b\star\kappa)\!\left(\frac ea\right)=
 \frac1a\PV\int_0^\infty
 \omega_b(u)
 \kappa\!\left(\frac{e}{au}\right)\frac{du}{u}\\
 &=
 \frac1a\PV\int_0^\infty
 \frac{a}{2\pi i}
 \frac{\kappa(e/t)}{t(a-t)}\,dt=
 \frac1a\PV\int_0^\infty G(t)\,dt.
\end{aligned}
\]

Now fix an integer \(K>a\).  Since
\(
  Q(t)=G(t)-\frac{\phi(a)}{a-t}
\),
the assumptions imply that \(Q\) is integrable on every finite
interval \((0,K)\).  Therefore
\(
  \Delta_Q(K)
  =
  \sum_{\substack{1\le c\le K\\c\ne a}}Q(c)
  -
  \int_0^KQ(t)\,dt
\)
is well defined.  Using
\(
  G(t)=Q(t)+\frac{\phi(a)}{a-t}
\)
with the same cutoff \(K\) in the sum and in the integral gives the
exact identity
\[
\begin{aligned}
 &\sum_{\substack{1\le c\le K\\c\ne a}}G(c)
 -
 \PV\int_0^KG(t)\,dt\\
 &\qquad=
 \Delta_Q(K)
 +
 \phi(a)
 \left[
   \sum_{\substack{1\le c\le K\\c\ne a}}
   \frac1{a-c}
   -
   \PV\int_0^K\frac{dt}{a-t}
 \right].
\end{aligned}
\]
The bracket can be evaluated explicitly.  First,
\[
 \sum_{\substack{1\le c\le K\\c\ne a}}
 \frac1{a-c}=
 \sum_{c=1}^{a-1}\frac1{a-c}
 +
 \sum_{c=a+1}^{K}\frac1{a-c}=
 H_{a-1}-H_{K-a}.
\]
On the other hand,
\[
 \PV\int_0^K\frac{dt}{a-t}=
 \lim_{\varepsilon\downarrow0}
 \left(
   \int_0^{a-\varepsilon}\frac{dt}{a-t}
   +
   \int_{a+\varepsilon}^{K}\frac{dt}{a-t}
 \right)=
 \log a-\log(K-a).
\]
Thus
\[
\sum_{\substack{1\le c\le K\\c\ne a}}
 \frac1{a-c}
 -
 \PV\int_0^K\frac{dt}{a-t}=
 H_{a-1}-H_{K-a}-\log a+\log(K-a).
\]
Since
\(
  H_N-\log N\longrightarrow\gamma
\)
and
\(
  \psi(a)=H_{a-1}-\gamma
\),
we obtain
\[
  H_{a-1}-H_{K-a}-\log a+\log(K-a)
  \longrightarrow
  \psi(a)-\log a.
\]

Finally, as \(K\to\infty\) through the integers, the truncated sum in
\(G\) converges to \(\sum_{c\ge1}'G(c)\), by absolute convergence,
and the truncated principal-value integral converges to
\(\PV\int_0^\infty G\), by the assumed integrability away from \(a\).
It follows at the same time that \(\Delta_Q(K)\) has a finite limit
and that
\[
  \bigl[
    W(\omega_b)W(\kappa)-W(\omega_b\star\kappa)
  \bigr]_{jl}
  =
  \frac1a
  \left[
    \lim_{K\to\infty}\Delta_Q(K)
    +
    \phi(a)\bigl(\psi(a)-\log a\bigr)
  \right].
\]

The common cutoff is essential: the sum and the integral involving
\(Q\) need not converge separately, whereas their common-cutoff
difference does.
\end{proof}

We now state the quantitative estimate used in the internal chain.
Let $(K_m)$ denote the condition
\begin{equation}\label{eq:Km}\tag{$K_m$}
  |\kappa(v)|+|v\kappa'(v)|+|v^2\kappa''(v)|
  \le
  K(1+|\log v|)^m\min(1,v^{-1}),
   v>0.
\end{equation}
{
\begin{theorem}[Intertwining defect]\label{thm:intertwining}
If $\kappa\in C^2(0,\infty)$ satisfies $(K_m)$, then $W(\kappa)$ is
bounded and
\begin{equation}\label{eq:intertwining}
  W(\omega_b)W(\kappa)-W(\omega_b\star\kappa)\in\mathcal S_2.
\end{equation}
More precisely, with $a=j+1$ and $e=l+1$,
\begin{equation}\label{eq:intertwining-entry}
  \left|
    [W(\omega_b)W(\kappa)-W(\omega_b\star\kappa)]_{jl}
  \right|
  \le
  C_mK
  \frac{(1+\log a+\log e)^{m+1}}{ae}.
\end{equation}
Every $\kappa\in\Finf_m$ satisfies $(K_m)$ by
Lemma~\ref{lem:euler}.
\end{theorem}
\begin{proof}
Condition $(K_m)$ implies $\kappa\in\FBV_m$ with no exceptional
points. Indeed, with $v=e^t$,
\[
  |f_\kappa(t)|
  =
  e^{t/2}|\kappa(v)|
  \le
  K\mu_m(t),
\]
while
\[
\begin{aligned}
  |f_\kappa'(t)|
  &=
  e^{t/2}
  \left|
    \frac12\kappa(v)+v\kappa'(v)
  \right| \\
  &\le
  e^{t/2}\bigl(|\kappa(v)|+|v\kappa'(v)|\bigr)
  \le
  K\mu_m(t).
\end{aligned}
\]
Thus Lemma~\ref{lem:regularity} and Lemma~\ref{lem:schur}
give boundedness of $W(\kappa)$.
Write
\[
  D
  =
  W(\omega_b)W(\kappa)-W(\omega_b\star\kappa),
\]
where, until boundedness is obtained at the end of the proof, the
second term denotes the sampled matrix of the pointwise
principal-value convolution.
We first verify the remaining convergence hypotheses of
Proposition~\ref{prop:pvpairing}. For large $t$, condition $(K_m)$
gives
\[
  |G(t)|
  \le
  C_{a,e,m}K(1+\log t)^m t^{-2},
\]
so the tail integral converges absolutely. As $t\to 0$, the
factor $\min(1,t/e)$ in $(K_m)$ yields
\[
  |G(t)|
  \le
  C_{a,e,m}K
  (1+\log(e/t))^m,
\]
which is integrable at the origin. Finally,
$\kappa\in C^2(0,\infty)$ makes $\phi$ continuously differentiable
near $a$, and
\[
  Q(t)
  =
  -\int_0^1
  \phi'\bigl(a+s(t-a)\bigr)\,ds
\]
is bounded near $a$. Hence Proposition~\ref{prop:pvpairing} applies.
Set
\[
  \beta=\frac ea,
  \qquad
  A_\beta(s)=\frac{\kappa(\beta/s)}{s},
  \qquad
  R_\beta(s)
  =
  \frac{A_\beta(s)-A_\beta(1)}{1-s}.
\]
We give $R_\beta$ its continuous value at $s=1$,
\[
  R_\beta(1)
  =
  -A_\beta'(1)
  =
  \kappa(\beta)+\beta\kappa'(\beta).
\]
Correspondingly, we give $Q$ its continuous value at $t=a$,
\[
  Q(a)
  =
  -\phi'(a)
  =
  \frac{R_\beta(1)}{2\pi ia},
\]
so that $Q$ has no jump at $a$.
Direct scaling in \eqref{eq:phiGQ} gives
\begin{equation}\label{eq:Qas}
  Q(as)
  =
  \frac{R_\beta(s)}{2\pi ia},
  \qquad
  \int_0^1|Q(t)|\,dt
  =
  \frac1{2\pi}
  \int_0^{1/a}|R_\beta(s)|\,ds,
\end{equation}
and, since $Q$ is continuous at $a$,
\[
  \Var_{[1,\infty)}Q
  =
  \frac1{2\pi a}
  \Var_{[1/a,\infty)}R_\beta.
\]
Near $s=1$, the divided difference has the integral representation
\[
  R_\beta(s)
  =
  -\int_0^1
  A_\beta'\bigl(1+\tau(s-1)\bigr)\,d\tau,
\]
and
\[
  R_\beta'(s)
  =
  -\int_0^1
  \tau
  A_\beta''\bigl(1+\tau(s-1)\bigr)\,d\tau.
\]
Since, with $v=\beta/s$,
\[
  A_\beta'(s)
  =
  -s^{-2}\bigl(\kappa(v)+v\kappa'(v)\bigr),
\]
and
\[
  A_\beta''(s)
  =
  s^{-3}
  \bigl(
    2\kappa(v)+4v\kappa'(v)+v^2\kappa''(v)
  \bigr),
\]
condition $(K_m)$ gives
\begin{equation}\label{eq:Rbeta-middle}
  \sup_{[1/2,2]}|R_\beta|
  +
  \Var_{[1/2,2]}R_\beta
  \le
  C_mK
  \frac{(1+|\log\beta|)^m}{\beta}.
\end{equation}
We now estimate the head and the two tails. From $(K_m)$,
\begin{equation}\label{eq:Abeta-bounds}
\begin{aligned}
  |A_\beta(s)|
  &\le
  K(1+|\log(\beta/s)|)^m
  \min\left(\frac1s,\frac1\beta\right),\\
  |A_\beta'(s)|
  &\le
  2K(1+|\log(\beta/s)|)^m
  \min\left(\frac1{s^2},\frac1{s\beta}\right).
\end{aligned}
\end{equation}
Set
\[
  L=1+\log a+\log e.
\]
Assume first that $a\ge2$. On $0<s\le1/a$ one has
$s\le\beta$ and $|1-s|\ge1/2$, and therefore
\[
  |R_\beta(s)|
  \le
  \frac{2K}{\beta}
  \left[
    (1+\log(\beta/s))^m
    +(1+|\log\beta|)^m
  \right].
\]
With $s=x/a$ and $a\beta=e$, this yields
\begin{equation}\label{eq:Rbeta-tail}
  \frac1a
  \int_0^{1/a}|R_\beta(s)|\,ds
  \le
  C_mK\frac{L^m}{ae}.
\end{equation}
For $a=1$, split $(0,1)$ at $1/2$: the preceding argument applies
on $(0,1/2]$, while \eqref{eq:Rbeta-middle} controls $[1/2,1)$.
Thus \eqref{eq:Rbeta-tail} holds for every $a\ge1$.
We next estimate the variation. Away from $s=1$, the quotient rule
gives
\[
  R_\beta'(s)
  =
  \frac{
    A_\beta'(s)(1-s)+A_\beta(s)-A_\beta(1)
  }{(1-s)^2}.
\]
Hence
\[
  |R_\beta'(s)|
  \le
  2|A_\beta'(s)|
  +
  4\bigl(|A_\beta(s)|+|A_\beta(1)|\bigr),
  \qquad
  \frac1a\le s\le\frac12,
\]
and
\[
  |R_\beta'(s)|
  \le
  \frac{2|A_\beta'(s)|}{s}
  +
  \frac{
    4\bigl(|A_\beta(s)|+|A_\beta(1)|\bigr)
  }{s^2},
  \qquad s\ge2.
\]
The principal derivative term is estimated by splitting at $s=\beta$:
\[
\begin{aligned}
  \int_{1/a}^\infty|A_\beta'(s)|\,ds
  &\le
  \frac{2K}{\beta}
  \left[
    \int_0^{\log e}(1+r)^m\,dr
    +
    \int_1^\infty
    \frac{(1+\log x)^m}{x^2}\,dx
  \right] \\
  &\le
  C_mK
  \frac{(1+\log e)^{m+1}}{\beta}.
\end{aligned}
\]
For $a\ge2$, the remaining terms on the two exterior regions
$[1/a,1/2]$ and $[2,\infty)$ satisfy
\[
  \int_{1/a}^{1/2}|A_\beta(s)|\,ds
  +
  \int_2^\infty|A_\beta(s)|s^{-2}\,ds
  +
  |A_\beta(1)|
  \le
  C_mK\frac{L^m}{\beta},
\]
again by splitting at $s=\beta$ and using
\eqref{eq:Abeta-bounds}. When $a=1$, the first exterior interval is
empty, and \eqref{eq:Rbeta-middle} together with the estimate on
$[2,\infty)$ gives the same conclusion.
Combining these estimates with \eqref{eq:Rbeta-middle} gives
\begin{equation}\label{eq:Rbeta-var}
  \Var_{[1/a,\infty)}R_\beta
  \le
  C_mK\frac{L^{m+1}}{\beta}.
\end{equation}
Finally, $(K_m)$ and the definition of $R_\beta(1)$ give
\begin{equation}\label{eq:Rbeta-point}
  |R_\beta(1)|+|\kappa(\beta)|
  \le
  C_mK\frac{L^m}{\beta}.
\end{equation}
We now estimate the common-cutoff term in
Proposition~\ref{prop:pvpairing}. For integers $M>a$, let
\[
  E_M(Q)
  =
  \sum_{c=1}^M Q(c)
  -
  \int_1^M Q(t)\,dt.
\]
Since $\Delta_Q(M)$ omits the lattice point $c=a$, the exact
finite-cutoff identity is
\begin{equation}\label{eq:DeltaQ}
  \Delta_Q(M)
  =
  -\int_0^1Q(t)\,dt
  -
  Q(a)
  +
  E_M(Q).
\end{equation}
The same cutoff $M$ is used in both the sum and the integral; in
particular, neither term is separated before taking their difference.
On each unit interval,
\[
  E_M(Q)
  =
  Q(M)
  +
  \sum_{c=1}^{M-1}
  \int_c^{c+1}
  \bigl(Q(c)-Q(t)\bigr)\,dt.
\]
Moreover,
\[
  \sum_{c\ge1}
  \left|
    \int_c^{c+1}
    \bigl(Q(c)-Q(t)\bigr)\,dt
  \right|
  \le
  \sum_{c\ge1}\Var_{[c,c+1]}Q
  =
  \Var_{[1,\infty)}Q
  <
  \infty.
\]
Thus the series converges absolutely.
It remains only to verify the boundary term. By
\eqref{eq:Abeta-bounds},
\[
  |A_\beta(s)|
  \le
  K\frac{(1+\log(s/\beta))^m}{s}
  \longrightarrow0,
  \qquad s\to\infty.
\]
Hence
\[
  R_\beta(s)
  =
  \frac{A_\beta(s)-A_\beta(1)}{1-s}
  \longrightarrow0,
\]
and, by \eqref{eq:Qas},
\[
  Q(M)
  =
  \frac{1}{2\pi ia}
  R_\beta\!\left(\frac Ma\right)
  \longrightarrow0.
\]
Consequently, $E_M(Q)$ has a finite limit and
\[
  \left|
    \lim_{M\to\infty}E_M(Q)
  \right|
  \le
  \Var_{[1,\infty)}Q.
\]
Combining \eqref{eq:DeltaQ} with \eqref{eq:Qas}, and using
\[
  |\psi(a)-\log a|
  \le
  \frac Ca,
  \qquad
  |\phi(a)|
  =
  \frac{|\kappa(\beta)|}{2\pi},
\]
Proposition~\ref{prop:pvpairing} gives
\[
  |D_{jl}|
  \le
  C
  \left[
    \frac1a
    \int_0^{1/a}|R_\beta(s)|\,ds
    +
    \frac1{a^2}
    \left(
      |R_\beta(1)|
      +
      \Var_{[1/a,\infty)}R_\beta
      +
      |\kappa(\beta)|
    \right)
  \right].
\]
Substituting \eqref{eq:Rbeta-tail}, \eqref{eq:Rbeta-var},
\eqref{eq:Rbeta-point}, and
\[
  a^{-2}\beta^{-1}=(ae)^{-1},
\]
we obtain
\[
  |D_{jl}|
  \le
  C_mK
  \frac{L^{m+1}}{ae},
\]
which is \eqref{eq:intertwining-entry}. Its square is summable over
$a,e\ge1$, and therefore
\[
  D\in\mathcal S_2.
\]
Finally, since $W(\omega_b)W(\kappa)$ is bounded,
\[
  W(\omega_b\star\kappa)
  =
  W(\omega_b)W(\kappa)-D
\]
is bounded as well.
\end{proof}
}

\part{Proofs of the main theorems}
\section{Proof of the Barr\'{\i}a--Halmos theorem}\label{sec:barria-halmos}

We can now assemble Theorem~\ref{thm:L3},
Proposition~\ref{prop:orbitdensity}, and
Theorem~\ref{thm:intertwining} into an exact module relation in the
Calkin algebra, which is the last step needed to prove
Theorem~\ref{thm:cesaro-main}. The crucial point is that
$\widetilde\vartheta \notin C_0(\mathbb R)$, so one cannot directly
write $\Lambda(\widetilde\vartheta)$; instead, multiplication by
$\widetilde\vartheta$ acts as a bounded multiplier of $C_0(\mathbb
R)$.

\begin{lemma}[Module relation]\label{lem:module}
For every $g \in C_0(\mathbb R)$,
\begin{equation}\label{eq:module}
  \pi(T_b)\,\Lambda(g) = \Lambda(\widetilde\vartheta\, g).
\end{equation}
Consequently, for every integer $n\ge0$,
\begin{equation}\label{eq:module-power}
  \pi(T_b)^n\, \Lambda(g) = \Lambda(\widetilde\vartheta^n g).
\end{equation}
\end{lemma}

\begin{proof}
We first prove~\eqref{eq:module} on the dense subalgebra 
$I =
\{p\circ\widetilde\vartheta : p\in\mathbb C[v],\ p(0)=p(1)=0\} \subset
C_0(\mathbb R)$ introduced in~\eqref{eq:Idef}. Let $g\in I$.
By Proposition~\ref{prop:orbitdensity}, there is a finite linear
combination of orbit kernels $\kappa_g = \sum_{r=0}^N a_r\kappa_r$ such
that $\sigma_{\kappa_g} = g$. Since each $\kappa_r \in F^\infty_{r+1}$,
the finite sum $\kappa_g$ belongs to $F^\infty_M$ for some $M$. Hence
Theorem~\ref{thm:intertwining} applies and gives $W(\omega_b)W(\kappa_g)
- W(\omega_b\star\kappa_g) \in \mathcal S_2$. Passing to the Calkin
algebra and using $T_b = W(\omega_b)+\tfrac12I$ give
\[
\begin{aligned}
      \pi(T_b)\Lambda(g) = \pi(T_b)\pi(W(\kappa_g))
  &= \pi\Bigl( \bigl( W(\omega_b)+\tfrac12I \bigr) W(\kappa_g) \Bigr)\\
  &= \pi\Bigl( W\Bigl( \omega_b\star\kappa_g + \tfrac12\kappa_g \Bigr) \Bigr).
\end{aligned}
\]
The last kernel can be identified directly from the orbit recursion:
since $\kappa_{r+1} = \omega_b\star\kappa_r + \tfrac12\kappa_r$,
$\omega_b\star\kappa_g + \tfrac12\kappa_g = \sum_{r=0}^N a_r\kappa_{r+1}$,
again a finite linear combination of orbit kernels. By~\eqref{eq:polemult},
$\sigma_{\omega_b\star\kappa_g+\frac12\kappa_g} = \widetilde\vartheta
\sigma_{\kappa_g} = \widetilde\vartheta g$. By the definition of
$\Lambda$, $\pi(T_b)\Lambda(g) = \Lambda(\widetilde\vartheta g)$ for
$g\in I$.

It remains to pass from $I$ to $C_0(\mathbb R)$. By
Proposition~\ref{prop:orbitdensity}, $I$ is uniformly dense
in $C_0(\mathbb R)$. Moreover, the maps $g \mapsto \pi(T_b)\Lambda(g)$
and $g\mapsto\Lambda(\widetilde\vartheta g)$ are continuous in the
uniform norm: the first because $\Lambda$ is bounded, the second
because $\|\widetilde\vartheta g\|_\infty \le \|\widetilde\vartheta\|_\infty
\|g\|_\infty \le \|g\|_\infty$. Thus~\eqref{eq:module} extends to every
$g\in C_0(\mathbb R)$. Finally, since $\widetilde\vartheta g \in
C_0(\mathbb R)$, we may apply~\eqref{eq:module} repeatedly; induction
gives~\eqref{eq:module-power}.
\end{proof}

\begin{theorem}[Internal Toeplitz realization]\label{thm:internalT}
The range of the Calkin Mellin calculus is contained in the Calkin
image of the Toeplitz algebra:
\begin{equation}\label{eq:internalT}
  \Lambda(C_0(\mathbb R))\subset \pi(\Toep).
\end{equation}
Consequently,
\[
  W(\kappa)\in\Toep,
  \qquad \kappa\in\A,
\]
and in particular \(C\in\Toep\).
\end{theorem}

\begin{proof}
Let \(p\in\mathbb C[v]\) satisfy
\(
  p(0)=p(1)=0
\),
and write
\(
  p(v)=v(v-1)q(v)\), \(
   q\in\mathbb C[v].
\)
Set
\(
  g=\sigma_{\omega_\Gamma^{\star2}}
   =\pi^2\operatorname{sech}^2(\pi\xi)
\).
By \eqref{eq:Gamma2symbol} and the definition of the Calkin Mellin
calculus in Theorem~\ref{thm:calkin-main},
\[
  \pi(\Gamma^2)
  =
  \pi\bigl(W(\omega_\Gamma^{\star2})\bigr)
  =
  \Lambda(g).
\]
Write
\(
  q(v)=\sum_{r=0}^N q_rv^r
\).
Using \eqref{eq:module-power} and linearity of \(\Lambda\), we obtain
\[
\begin{aligned}
  \pi(q(T_b))\pi(\Gamma^2)
  &=
  \sum_{r=0}^N q_r\,\pi(T_b^r)\Lambda(g)=
  \sum_{r=0}^N q_r\,\Lambda(\thetaM^r g)=
  \Lambda\bigl(q(\thetaM)g\bigr).
\end{aligned}
\]
Now Theorem~\ref{thm:L3} gives
\[
  p(T_b)-T_{p\circ b}
  =
  -\frac1{4\pi^2}q(T_b)\Gamma^2+K_p,
  \qquad K_p\in\mathcal S_2.
\]
Passing to the Calkin algebra therefore yields
\[
\begin{aligned}
  \pi\bigl(p(T_b)-T_{p\circ b}\bigr)
  &=
  -\frac1{4\pi^2}
  \Lambda\bigl(q(\thetaM)g\bigr)=
  \Lambda\bigl(p\circ\thetaM\bigr),
\end{aligned}
\]
because, by \eqref{eq:Gamma2symbol},
\(
  {g}/{4\pi^2}
  =
  \thetaM(1-\thetaM)
\),
and hence
\[
  -\frac1{4\pi^2}q(\thetaM)g
  =
  -q(\thetaM)\thetaM(1-\thetaM)
  =
  \thetaM(\thetaM-1)q(\thetaM)
  =
  p(\thetaM).
\]

Since \(T_b\in\Toep\), we have \(p(T_b)\in\Toep\). Moreover
\(T_{p\circ b}\in\Toep\).  Thus
\[
  \Lambda(p\circ\thetaM)\in\pi(\Toep).
\]
Consequently,
\[
  \Lambda(I)\subset\pi(\Toep),
\]
where \(I\) is the dense subalgebra defined in
Proposition~\ref{prop:orbitdensity}.  Since \(I\) is dense in
\(C_0(\mathbb R)\), \(\Lambda\) is continuous, and
\(\pi(\Toep)\) is a closed \(C^*\)-subalgebra of the Calkin algebra,
we obtain
\[
  \Lambda(C_0(\mathbb R))\subset\pi(\Toep).
\]

Finally, let \(\kappa\in\A\).  By Theorem~\ref{thm:calkin-main},
\[
  \Lambda(\sigma_\kappa)=\pi(W(\kappa)).
\]
By the inclusion just proved there exists \(A\in\Toep\) such that
\[
  \pi(A)=\Lambda(\sigma_\kappa)=\pi(W(\kappa)).
\]
Hence
\[
  W(\kappa)-A\in\Comp.
\]
Since \(\Comp\subset\Toep\), it follows that
\[
  W(\kappa)\in\Toep.
\]
Taking \(\kappa=\omega_C\), for which \(W(\omega_C)=C\), gives
\(C\in\Toep\).
\end{proof}

The preceding theorem already resolves the original membership
question. The following elementary observation upgrades it to the
commutator ideal and proves Theorem~\ref{thm:cesaro-main} in the form
stated in the introduction.

\begin{corollary}[Barr\'{\i}a--Halmos theorem]\label{cor:barriahalmos}
The classical Ces\`aro operator belongs to the commutator ideal
$\Comm$ of the Toeplitz algebra.
\end{corollary}
\begin{proof}
By Theorem~\ref{thm:internalT}, \(C\in\Toep\). The exact matrix identity
\begin{equation}\label{eq:CT1-z}
  CT_{1-z}
  =
  \operatorname{diag}\left(\frac{1}{n+1}\right)_{n\ge0}
\end{equation}
is compact.
Let
\[
  \varsigma:\Toep\longrightarrow L^\infty(\mathbb T)
\]
be the Toeplitz symbol homomorphism. Its kernel is the commutator ideal
\(\Comm\), and
\(
  \Comp\subset\Comm=\ker\varsigma
\),
see~\cite{Douglas}. Hence the compact operator on the right-hand side
of \eqref{eq:CT1-z} is annihilated by \(\varsigma\). Applying
\(\varsigma\) to \eqref{eq:CT1-z} therefore gives
\[
  \varsigma(C)(1-z)=0
  \qquad\text{a.e. on }\mathbb T.
\]
Since \(1-z\neq0\) almost everywhere on \(\mathbb T\), it follows that
\(
  \varsigma(C)=0
\).
Thus
\(
  C\in\ker\varsigma=\Comm
\).
\end{proof}

The fact that if $C \in \mathcal{T}$ then $C \in \Comm$, was already explained in \cite{BH} and discussed in \cite{BS}.

At this point, the independence from Sang's result becomes
formal: the proof of Corollary~\ref{cor:barriahalmos} uses only the
original results of the present paper.

\section{The full Mellin--Toeplitz calculus}\label{sec:full-calculus}

Having proved the Barr\'{\i}a--Halmos theorem, we now collect the
broader consequences of the machinery built for its proof, which
together give the complete statement of Theorem~\ref{thm:intro-main}
announced in the introduction. Let $\varsigma : \Toep \to
L^\infty(\mathbb T)$ again denote the Toeplitz symbol homomorphism,
with kernel $\Comm$ and $\Comp \subset \Comm$~\cite{Douglas}.

\begin{theorem}\label{thm:full-calculus}
For the class $\Kernels = \bigcup_{m\ge0} \FBV_m$ the conclusions of
Theorem~\ref{thm:intro-main} hold. More explicitly:
\begin{enumerate}[label=(\roman*)]
  \item $\Kernels$ is a $\vee$-closed commutative complex
    $\star$-algebra and every $W(\kappa)$ is bounded;
  \item for $\kappa,\eta \in \Kernels$, $W(\kappa)W(\eta) -
    W(\kappa\star\eta) \in \mathcal S_2$;
  \item $\Lambda$ is the isometric $*$-isomorphism of
    Theorem~\ref{thm:calkin-main};
  \item every $W(\kappa)$ belongs to $\Comm \subset \Toep$;
  \item $\|W(\kappa)\|_{\mathrm{ess}} = \|\sigma_\kappa\|_\infty$,
    $\spec_{\mathrm{ess}}(W(\kappa)) = \overline{\sigma_\kappa(\mathbb R)}$;
  \item $[W(\kappa),W(\eta)]$ and $[W(\kappa)^*,W(\kappa)]$ are
    Hilbert--Schmidt.
\end{enumerate}
\end{theorem}

\begin{proof}
The algebra, involution, and boundedness assertions are
Proposition~\ref{prop:closure}, Lemma~\ref{lem:regularity}, and
Lemma~\ref{lem:schur}; the Hilbert--Schmidt product formula is
Theorem~\ref{thm:HSproduct}; and the Calkin identification, essential
norm, and essential spectrum are Theorem~\ref{thm:calkin-main}. The
internal Toeplitz inclusion is Theorem~\ref{thm:internalT}.

It remains to upgrade membership from $\Toep$ to $\Comm$ for the whole
class. By Corollary~\ref{cor:barriahalmos}, $\varsigma(C)=0$. Because
$\Comp \subset \Comm = \ker\varsigma$, the symbol map factors through
the Calkin image of the Toeplitz algebra: there is a
$*$-homomorphism $\overline\varsigma : \pi(\Toep) \to L^\infty(\mathbb
T)$, $\overline\varsigma(\pi(A)) = \varsigma(A)$. From
Theorem~\ref{thm:calkin-main}, $\pi(W(\kappa)) = F_\kappa(\pi(C))$,
$F_\kappa(0)=0$. {A $*$-homomorphism commutes with continuous
functional calculus}, and $\overline\varsigma(\pi(C)) = \varsigma(C) =
0$; hence
\[
  \varsigma(W(\kappa)) = \overline\varsigma\bigl( F_\kappa(\pi(C)) \bigr)
  = F_\kappa\bigl( \overline\varsigma(\pi(C)) \bigr) = F_\kappa(0) = 0.
\]
Thus $W(\kappa)\in\Comm$.

Finally, apply the product formula to $(\kappa,\eta)$ and
$(\eta,\kappa)$ and subtract; since $\star$ is commutative, this gives
$[W(\kappa),W(\eta)]\in\mathcal S_2$. Taking $\eta=\kappa^\vee$ and
using $W(\kappa)^*=W(\kappa^\vee)$ gives the self-commutator
assertion.
\end{proof}

\begin{corollary}[Ces\`aro and Hilbert matrices]\label{cor:cesaro-hilbert}
The Ces\`aro operator belongs to the commutator ideal of the Toeplitz
algebra, and
\[
  \|C\|_{\mathrm{ess}} = 2, \qquad \spec_{\mathrm{ess}}(C) = \{w : |w-1|=1\}.
\]
For the classical Hilbert matrix $\Gamma$,
\[
  \|\Gamma\|_{\mathrm{ess}} = \pi, \qquad \spec_{\mathrm{ess}}(\Gamma) = [0,\pi].
\]
\end{corollary}

\begin{proof}
For $C$, use $\omega_C \in \FBV_0$, $W(\omega_C)=C$, and
$\sigma_{\omega_C}(\xi) = (1/2+i\xi)^{-1}$. For the Hilbert matrix, use
the shifted kernel $\omega_\Gamma$, whose Mellin symbol is $\pi
\operatorname{sech}(\pi\xi)$; the shift from $(j+k+2)^{-1}$ to
$(j+k+1)^{-1}$ is Hilbert--Schmidt and hence invisible in the Calkin
algebra.
\end{proof}

Corollary~\ref{cor:cesaro-hilbert} is precisely the instance mentioned
in Section~\ref{sec:sang}: it is a family-level statement --- about
$\Gamma$, not about $C$ --- that follows immediately from
Theorem~\ref{thm:full-calculus} but has no counterpart in Sang's
single-operator approach.

\section{Concluding remarks}\label{sec:concluding}

We began with the Barr\'{\i}a--Halmos question of whether the
Ces\`aro operator belongs to the Toeplitz algebra. The proof led to a
larger structure: the sampled-ratio map $\kappa \mapsto W(\kappa)$ is
a discrete Mellin quantization whose failure of multiplicativity is
Hilbert--Schmidt. The regular part of the argument produces a Mellin
symbol calculus in the Calkin algebra, while the singular sawtooth
kernel provides the mechanism that connects this abstract calculus
with the Toeplitz algebra. The Ces\`aro operator is the kernel
$\mathbf 1_{(0,1]}$ inside this framework.

From this perspective, the Toeplitz membership of $C$ is both the
original problem and the first instance of the general theorem. The
quantitative lattice defect is the part of the construction that has
no analogue in the exact continuous Mellin model and is likely to
persist in other discrete settings. Natural questions include Schatten
refinements of the defect, extensions to $\ell^p$ and weighted
sequence spaces, and a localization theory that identifies the sampled
Mellin fibre directly from the Toeplitz side.

\section*{Methodology}
For the writing and preparation of this article, the authors used the large language models ChatGPT and Claude, including for proofreading and bibliographical research. The authors remain fully responsible for the content of the paper and adhere to the principles of the “Leiden Declaration on Artificial Intelligence and Mathematics.”



\begin{thebibliography}{99}

\bibitem{BH} J. Barr\'{\i}a and P. R. Halmos, \emph{Asymptotic
Toeplitz operators}, Trans.\ Amer.\ Math.\ Soc.\ \textbf{273} (1982),
no.\ 2, 621--630.

\bibitem{BS}C. Bellavita, G. Stylogianns, \emph{On asymptotic and essential Toeplitz and Hankel integral operator}, 	arXiv:2409.10014 [math.FA].

\bibitem{BDS} C. Bellavita, E.Dellepiane,  G. Stylogianns, \emph{Boundedness, compactness and Schatten class for Rhaly matrices}, J. Lond. Math. Soc. (2) 112 (2025), no. 4, Paper No. e70304, 36 pp.

\bibitem{BrownHalmos} A. Brown and P. R. Halmos, \emph{Algebraic
properties of Toeplitz operators}, J.\ Reine Angew.\ Math.\
\textbf{213} (1963/64), 89--102.

\bibitem{BHS} A. Brown, P. R. Halmos and A. L. Shields, \emph{Ces\`aro
operators}, Acta Sci.\ Math.\ (Szeged) \textbf{26} (1965), 125--137.

\bibitem{BSK} A. B\"ottcher, B. Silbermann and A. Yu. Karlovich,
\emph{Analysis of Toeplitz Operators}, 2nd ed., Springer, 2006.

\bibitem{Douglas} R. G. Douglas, \emph{Banach Algebra Techniques in
Operator Theory}, Pure and Applied Mathematics 49, Academic Press,
1972.

\bibitem{Duduchava} R. Duduchava, \emph{Integral Equations with Fixed
Singularities}, Teubner-Texte zur Mathematik, Leipzig, 1979.

\bibitem{DK03}H. Dym and V. Katsnelson, \emph{Contributions of Issai Schur to Analysis}, in Studies in Memory of Issai Schur, A. Joseph, A. Melnikov and R. Rentschler (eds.), Progress in Mathematics, Birkhäuser, 2003.

\bibitem{F} A. Feintuch,
\emph{On asymptotic Toeplitz and Hankel operators}, in The Gohberg Anniversary Collection: Volume I: The Calgary Conference and Matrix Theory Papers and Volume II: Topics in Analysis and Operator Theory, 
Birkh{\"a}user Basel, 1989,
733--746.


\bibitem{GK} I. Ts. Gohberg and N. Ya. Krupnik, \emph{On the algebra
generated by Toeplitz matrices}, Functional Anal.\ Appl.\ \textbf{3}
(1969), 119--127.

\bibitem{MR} J. Mashreghi and W. T. Ross, \emph{The Wonders of the
Ces\`aro Operator}, Operator Theory: Advances and Applications 311,
Birkh\"auser, Cham, 2026.

\bibitem{Plaschinsky} P. V. Plaschinsky, \emph{Discrete Mellin
convolution with dilation and its applications}, Math.\ Model.\
Anal.\ \textbf{3} (1998), no.\ 1, 160--167.

\bibitem{Rudin} W. Rudin, \emph{Fourier Analysis on Groups},
Interscience, 1962.

\bibitem{Sang} Y. Sang, \emph{The Ces\`aro operator is in the
Toeplitz algebra}, arXiv:2608.25740v1 [math.FA], 26 August 2026.

\bibitem{Schur04} I. Schur, \emph{Remarks on the theory of bounded bilinear forms with infinitely many variables }, Journ.\ für\ reine\ und\ angew.\ Math.\  \textbf{140} (1911), 1 --28.


\bibitem{S87} A. G. Siskakis, \emph{Composition semigroups and the Cesàro operator on  $H^p$},
J. London Math. Soc. (2) 36 (1987), no. 1, 153–164.


\bibitem{Young2004} 
 Scott W. Young,  \emph{Spectral properties of generalized Cesàro operators},
Integral Equations Operator Theory 50 (2004), no. 1, 129–146.
\end{thebibliography}
\end{document}